\documentclass[12pt, letter, reqno]{amsart}
\usepackage{amsmath,amssymb,comment}

\makeatletter
\def\th@plain{%
  \thm@notefont{}
  \itshape 
}
\def\th@definition{%
  \thm@notefont{}
  \normalfont 
}
\makeatother

\usepackage{color}

\usepackage{enumerate}

\title[Helmholtz equation]{Sharp high-frequency estimates for the Helmholtz equation and applications to boundary integral equations}
\author{Dean Baskin}
\author{Euan Spence}
\author{Jared Wunsch}

\newtheorem{theorem}{Theorem}[section]
\newtheorem{lemma}[theorem]{Lemma}

\newtheorem{assumption}[theorem]{Assumption}
\newtheorem{definition}[theorem]{Definition}
\newtheorem{remark}[theorem]{Remark}
\newtheorem{corollary}[theorem]{Corollary}

\newcommand{\cF}{{\mathcal F}}

\newcommand{\cO}{{\mathcal O}}

\newcommand{\bx}{x}

\newcommand{\bn}{n}

\newcommand{\by}{y}

\DeclareMathOperator{\supp}{supp} 

\newcommand{\tOmega}{\widetilde{\Omega}}

\newcommand{\re}{{\rm e}}
\newcommand{\ri}{{\rm i}}
\newcommand{\rd}{{\rm d}}

\newcommand{\noi}{\noindent}

\newcommand{\beq}{\begin{equation}}
\newcommand{\eeq}{\end{equation}}
\newcommand{\beqs}{\begin{equation*}}
\newcommand{\eeqs}{\end{equation*}}
\newcommand{\bit}{\begin{itemize}}
\newcommand{\eit}{\end{itemize}}
\newcommand{\ben}{\begin{enumerate}}
\newcommand{\een}{\end{enumerate}}
\newcommand{\bal}{\begin{align}}
\newcommand{\eal}{\end{align}}
\newcommand{\bals}{\begin{align*}}
\newcommand{\eals}{\end{align*}}
\newcommand{\bse}{\begin{subequations}}
\newcommand{\ese}{\end{subequations}}
\newcommand{\bpr}{\begin{proposition}}
\newcommand{\epr}{\end{proposition}}
\newcommand{\bre}{\begin{remark}}
\newcommand{\ere}{\end{remark}}
\newcommand{\bpf}{\begin{proof}}
\newcommand{\epf}{\end{proof}}
\newcommand{\ble}{\begin{lemma}}
\newcommand{\ele}{\end{lemma}}
\newcommand{\bco}{\begin{corollary}}
\newcommand{\eco}{\end{corollary}}
\newcommand{\bex}{\begin{example}}
\newcommand{\eex}{\end{example}}
\newcommand{\bth}{\begin{theorem}}
\newcommand{\enth}{\end{theorem}}

\newcommand{\Rea}{\mathbb{R}}

\newcommand{\cond}{\mathop{{\rm cond}}}

\newcommand{\comp}{\mathop{{\rm comp}}}

\newcommand{\Oi}{{\Omega_-}}

\newcommand{\Oe}{{\Omega_+}}
\newcommand{\OR}{{\Omega_R}}

\newcommand{\eps}{\varepsilon}

\newcommand{\pdiff}[2]{\frac{\partial #1}{\partial #2}}
\newcommand{\dudn}{\pdiff{u}{n}}

\newcommand{\dnu}{\partial_n u}

\newcommand{\dnpu}{\partial_n^+ u}
\newcommand{\dnmu}{\partial_n^- u}

\newcommand{\dnpw}{\partial_n^+ w}

\newcommand{\dudnw}{\partial u/\partial n}

\newcommand{\nus}{|u|^2}
\newcommand{\ngus}{|\nabla u|^2}

\newcommand{\gu}{\nabla u}

\newcommand{\gv}{\nabla v}

\newcommand{\nT}{\nabla_\Gamma}

\newcommand{\half}{\frac{1}{2}}

\newcommand{\LtG}{L^2(\Gamma)}
\newcommand{\HoG}{H^1(\Gamma)}
\newcommand{\HhG}{H^{1/2}(\Gamma)}
\newcommand{\HmhG}{H^{-1/2}(\Gamma)}

\newcommand{\la}{\lambda}

\newcommand{\LtO}{L^2(\Omega)}

\newcommand{\Holoc}{H^1_{\text{loc}}}
\newcommand{\Holoce}{H^1_{\text{\emph{loc}}}}

\newcommand{\Lt}[1]{L^2( #1 )}
\newcommand{\Ho}[1]{H^1( #1 )}
\newcommand{\Hh}[1]{H^{1/2}( #1)}
\newcommand{\Hmh}[1]{H^{-1/2}( #1)}

\newcommand{\tendi}{\rightarrow \infty}

\newcommand{\opA}{A'_{k,\eta}}
\newcommand{\opAt}{\widetilde{A}'_{k,\eta}}

\newcommand{\opBM}{B_{k,\eta}}
\newcommand{\opBMt}{\widetilde{B}_{k,\eta}}
\newcommand{\opAinv}{(A'_{k,\eta})^{-1}}

\newcommand{\opBMinv}{B^{-1}_{k,\eta}}
\newcommand{\opBMtinv}{\widetilde{B}^{-1}_{k,\eta}}
\newcommand{\normA}{\|\opA\|}

\newcommand{\normAinv}{\|\opAinv\|}

\newcommand*{\N}[1]{\left\|#1\right\|}

\newcommand{\HokOe}{{H^1_k(\Oe)}}
\newcommand{\HokG}{{H^1_k(\Gamma)}}

\newcommand{\DtN}{P^+_{DtN}}
\newcommand{\NtD}{P^+_{NtD}}
\newcommand{\ItD}{P^{-,\eta}_{ItD}}
\newcommand{\ItN}{P^{-,\eta}_{ItN}}
\newcommand{\ItDR}{P^{-,\eta, R}_{ItD}}

\newcommand{\LtGt}{L^2(\Gamma) \rightarrow \LtG}

\newcommand{\ton}{\text{ on }}
\newcommand{\tin}{\text{ in }}
\newcommand{\tfa}{\text{ for all }}

\newcommand{\tand}{\text{ and }}

 \usepackage{hyperref}

\usepackage[margin=3cm]{geometry}

\renewcommand{\Im}{\operatorname{Im}}
\renewcommand{\Re}{\operatorname{Re}}

\newcommand{\notalpha}{a}
\newcommand{\notbeta}{b}

\newcommand{\Norm}[2][]{\left\| #2\right\|_{#1}}

\newcommand{\lap}{\Delta}
\newcommand{\pd}[1][]{\partial_{#1}}
\newcommand{\reals}{\mathbb{R}}
\newcommand{\dom}{\mathcal{D}}
\newcommand{\bpm}{\begin{pmatrix}}
\newcommand{\epm}{\end{pmatrix}}

\newcommand{\loc}{{\text{loc}}}
\newcommand{\pa}{\partial}
\newcommand{\abs}[1]{{\left\lvert{#1}\right\rvert}}
\newcommand{\norm}[1]{{\left\lVert{#1}\right\rVert}}
\newcommand{\ang}[1]{{\left\langle{#1}\right\rangle}}
\newcommand{\smallabs}[1]{{\lvert{#1}\rvert}}
\newcommand{\smallnorm}[1]{{\lVert{#1}\rVert}}
\newcommand{\smallang}[1]{{\langle{#1}\rangle}}
\newcommand{\Lap}{\Delta}
\newcommand{\RR}{\mathbb{R}}

\newcommand{\NN}{\mathbb{N}}

\newcommand{\CC}{\mathbb{C}}

\newcommand{\hyp}{\mathcal{H}}
\newcommand{\gla}{\mathcal{G}}

\newcommand{\Rimp}{R_{I,\eta}}
\newcommand{\tRimp}{\widetilde{R}_{I,\eta}}

\usepackage{color}

\begin{document}

\maketitle


\begin{abstract}
  We consider three problems for the Helmholtz equation in interior and
  exterior domains in $\mathbb{R}^d,$ ($d=2,3$): the exterior
  Dirichlet-to-Neumann and Neumann-to-Dirichlet problems for outgoing
  solutions, and the interior impedance problem.  We derive sharp estimates for
  solutions to these problems that, in combination, give bounds on the
  inverses of the combined-field boundary integral operators for exterior Helmholtz problems.
\end{abstract}

\section{Introduction}\label{sec:introduction}

Proving bounds on solution of the Helmholtz equation
\beq\label{eq:Helm_intro}
\Delta u + k^2 u = -f
\eeq
(where $f$ is a given function and $k\in \Rea\setminus\{0\}$ is the wavenumber) has a long
history. Nevertheless, the following problems have remained
open.
\begin{enumerate}[(i)]
\item Proving sharp bounds on the \emph{Dirichlet-to-Neu\-mann (DtN)} or
  \emph{Neu\-mann-to-Dirichlet (NtD)} maps for outgoing
  solutions of the homogeneous Helmholtz equation (i.e., equation~\eqref{eq:Helm_intro} with $f=0$) in exterior nontrapping domains.
\item Proving sharp bounds on the solution of the \emph{interior impedance problem (IIP)} for general domains, where this boundary value problem (BVP) consists of \eqref{eq:Helm_intro} posed in a bounded domain with the boundary condition
\beq\label{eq:imp_bc}
\dudn - \ri \eta u = g
\eeq
where $g$ is a given function and $\eta
\in \Rea\setminus\{0\}$.
\end{enumerate}
This paper fills these gaps in the literature.

The motivation for considering the exterior DtN and NtD maps for the Helmholtz equation is fairly clear, since these are natural objects to study in relation to scattering problems.
The motivation for studying the IIP is two-fold:
\bit
\item[(i)] It has become a standard model problem used when designing
  numerical methods for solving the Helmholtz equation (see
  Section~\ref{sec:iip-mot} below for further explanation), and to
  prove error
  estimates 
  one needs bounds on the solution of the BVP.
\item[(ii)] The integral equations used to solve the exterior Dirichlet, Neumann, and impedance problems can also be used to solve the IIP; therefore, to prove bounds on the inverses of these integral operators, one needs to have bounds on the solution of the IIP  -- we discuss this more in \S\ref{sec:int} below.
\eit

This paper may be regarded as a sequel to \cite{ChMo:08} and
\cite{Sp:14} as it variously sharpens and generalizes estimates
obtained in those works.  We will refer to these papers for many of the
basic results.  Although the results proved here hold for any
dimension $d \geq 2$, we state them only in dimensions $2$ and $3$, 
firstly since these are the most interesting for applications, and secondly since this avoids
re-proving background material only stated in the literature in these low dimensions.

\subsection{Statement of the main results}\label{sec:1-1}

Let $\Oi\subset \Rea^d, \,d=2,3,$ be a bounded, Lipschitz open set with boundary $\Gamma:=\partial \Oi$,
such that the open complement $\Oe:= \Rea^d \setminus \overline{\Oi}$ is connected.
Let $\gamma_{\pm}$ denote the trace operators from $\Omega_{\pm}$ to $\Gamma$, 
let $\partial_n^{\pm}$ denote the normal derivative trace operators, and let $\nT$ denote the surface gradient operator on $\Gamma$. 
Let $B_R:=\{\bx :|\bx|<R\}$.



\begin{definition}[Nontrapping]\label{def:nt1}
We say that $\Oe\subset \Rea^d, \,d=2, 3$ is 
\emph{nontrapping} if $\Gamma$ is smooth ($C^\infty$) and,
given $R> \sup_{\bx\in\Oi}|\bx|$, there exists a $T(R)<\infty$ such that 
all the billiard trajectories (in the sense of Melrose--Sj{\"o}strand~\cite{MeSj:82})
that start in $\Oe\cap B_R$ at time zero leave $\Oe\cap B_R$ by time $T(R)$.
\end{definition}

\begin{definition}[Nontrapping polygon]\label{def:nt2}
If $\Oi\subset \Rea^2$ is a polygon we say that it is a \emph{nontrapping polygon} if (i) no three vertices are collinear, and (ii), 
given $R> \sup_{\bx\in\Oi}|\bx|$, there exists a $T(R)<\infty$ such that 
all the billiard trajectories that start in $\Oe\cap B_R$ at time zero and miss the vertices leave $\Oe\cap B_R$ by time $T(R)$.
(For a more precise statement of (ii) see \cite[\S5]{BaWu:13}.)
\end{definition}

\begin{definition}[Star-shaped]\label{def:star} Let $\Oi\subset\Rea^d, \,d=2, 3,$ be a bounded, Lipschitz open set.

\noi(i) we say that $\Oi$ is \emph{star-shaped} if 
$\bx \cdot\bn(\bx)\geq 0$ for every $\bx \in \Gamma$ for which $\bn(\bx)$ is defined (where $n(x)$ is the normal to $\bx\in\Gamma$).

\noi(ii) we say that $\Oi$ is \emph{star-shaped with respect to a ball} if there exists a constant $c>0$ such that $\bx \cdot\bn(\bx)\geq c$ for every $\bx \in \Gamma$ for which $\bn(\bx)$ is defined. 

\end{definition}

\begin{theorem}[Bounds on the exterior DtN map]
\label{thm:1}
Let 
$u\in\Holoce(\Oe)$ satisfy the Helmholtz equation 
\beq\label{eq:Helm}
\Delta u +k^2u=0 \quad\mbox{ in } \Oe
\eeq
for $k \in \reals\setminus\{0\}$ and the Sommerfeld radiation condition 
\beq\label{eq:src}
\pdiff{u}{r}
- \ri k u
= o\left(\frac{1}{r^{(d-1)/2}}\right) 
\eeq
as $r:=|\bx| \tendi$, uniformly in $\hat{\bx}:=\bx/r$.
If \emph{either} $\Oe$ is nontrapping (in the sense Definition \ref{def:nt1}) \emph{or} $\Oi$ is a nontrapping polygon (in the sense of Definition \ref{def:nt2}) \emph{or} $\Oi$ is Lipschitz and star-shaped (in the sense of Definition \ref{def:star}(i)), 
then, given $k_0>0$,
\beq\label{eq:DtN1}
\N{\dnpu}_{\Hmh{\Gamma}} \lesssim \abs{k}\N{\gamma_+ u }_{\Hh{\Gamma}},
\eeq
for all $\abs{k}\geq k_0$.
Furthermore, if $\gamma_+ u \in \HoG$ then $\dnpu\in \LtG$ and, given $k_0>0$,
\beq\label{eq:DtN2}
\N{\dnpu}_{\Lt{\Gamma}} \lesssim \N{\nT (\gamma_+ u)}_{\Lt{\Gamma}} + \abs{k} \N{\gamma_+ u }_{\Lt{\Gamma}}
\eeq
for all $\abs{k}\geq k_0$.
\end{theorem}

\begin{theorem}[Bounds on the NtD map]\label{thm:2}
Let $\Oe$ be nontrapping (in the sense Definition \ref{def:nt1}) and let $u\in \Holoce(\Oe)$ satisfy the Helmholtz equation \eqref{eq:Helm} and the Sommerfeld radiation condition \eqref{eq:src}.
Let $\beta=2/3$ in the case when $\Gamma$ has strictly positive curvature, and $\beta=1/3$ otherwise.

Then, given $k_0>0$,
\beq\label{eq:NtD1}
\N{\gamma_+ u }_{\Hh{\Gamma}} \lesssim \abs{k}^{1-\beta} \N{\dnpu}_{\Hmh{\Gamma}},
\eeq
for all $\abs{k}\geq k_0$. Furthermore, if $\dnpu \in \LtG$ then $\gamma_+ u \in \HoG$ and, given $k_0>0$,
\beq\label{eq:NtD2}
\N{\nT (\gamma_+ u)}_{\Lt{\Gamma}} + \abs{k} \N{\gamma_+ u }_{\Lt{\Gamma}} \lesssim \abs{k}^{1-\beta} \N{\dnpu}_{\Lt{\Gamma}},
\eeq
for all $\abs{k}\geq k_0$.
\end{theorem}

By considering the specific examples of $\Gamma$ the unit circle (in 2-d) and the unit sphere (in 3-d) and using results about the asymptotics of Bessel and Hankel functions, it was shown in \cite[Lemmas 3.10, 3.12]{Sp:14} that the bounds \eqref{eq:DtN1} and \eqref{eq:DtN2} are sharp, and that \eqref{eq:NtD1} and \eqref{eq:NtD2} are sharp in the case of strictly positive curvature.

We prove the DtN bound \eqref{eq:DtN2} and can then get a bound on the DtN map between a range of Sobolev spaces by interpolation. Of this range, the bound \eqref{eq:DtN1} is the most interesting (since it is between the natural trace spaces for solutions of the Helmholtz equation) and thus we state it explicitly; similarly for \eqref{eq:NtD2} and \eqref{eq:NtD1}.

Our next result concerns the IIP under the following assumption about
the impedance parameters $\eta$.
We permit a more general assumption on $\eta$ than that specified in the
introduction: it can be variable, and need only have nonzero real part
with a linear rate of growth in
$k.$

\begin{assumption}[A particular class of $\eta$]\label{ass:eta}
$\eta(x) := \notalpha(x)k + \ri \notbeta(x)$ where 
$\notalpha, \notbeta$ are real-valued $C^\infty$ functions on $\Gamma$, $\notbeta\geq 0$ on $\Gamma$, and there exists an $\notalpha_->0$ such that \emph{either}
\beqs
\notalpha(x) \geq \notalpha_- >0 \quad \tfa x \in \Gamma \text{ \emph{or} } -\notalpha(x) \geq \notalpha_- >0 \quad \tfa x \in \Gamma.
\eeqs
\end{assumption}

For purposes of obtaining estimates valid down to $k=0$ (and in
particular, to make contact with applications in the work of
Epstein, Greengard, and Hagstrom~\cite{EpGrHa:15}) we will also state another, stronger, set
of hypotheses on $\eta.$

\begin{assumption}[Another class of $\eta$]\label{ass:eta2}
$\eta(x) := \notalpha(x)k + \ri \notbeta(x)$ where 
$\notalpha, \notbeta$ are real-valued $C^\infty$ functions on $\Gamma$ and there exists $\notalpha_->0$, $\notbeta_->0$ such that
\beqs
\notalpha(x) \geq \notalpha_- >0 \quad \tfa x \in \Gamma \text{ and } \notbeta(x) \geq \notbeta_- >0 \quad \tfa x \in \Gamma.
\eeqs
\end{assumption}

In our discussion of the impedance problem, we use $\Omega$ to denote
the domain where the IIP is posed (instead of $\Oi$), since we do not
need the restriction that we imposed on $\Oi$ that the open complement
is connected.

\bth[Bounds on the solution to the interior impedance problem]\label{thm:3}
Let $\Omega$ be a bounded $C^\infty$ open set in 2- or 3-dimensions with boundary $\Gamma$.
Given $g\in \LtG$, $f\in\LtO$, and $\eta$ satisfying Assumption \ref{ass:eta}, let $u \in H^1(\Omega)$ be 
be the solution to
the interior impedance problem 
\beq\label{eq:IIP}
\Delta u + k^2 u = -f \quad\mbox{ in } \Omega\quad\tand\quad
\dnu - \ri \eta \gamma u = g \quad \mbox{ on }\Gamma.
\eeq
Then
\beq\label{eq:thm3}
\N{\gu}_{\LtO} + \abs{k} \N{u}_{\LtO}
\lesssim
\N{f}_{\LtO}+ \N{g}_{\LtG}
\eeq
for all $k \in \RR.$ If the stronger Assumption~\ref{ass:eta2} holds,
estimate \eqref{eq:thm3} holds with $1+\abs{k}$ replacing $\abs{k}.$
\end{theorem}

The bound \eqref{eq:thm3} is sharp. Indeed, in \cite[Lemma
4.12]{Sp:14} it was proved that given any bounded Lipschitz domain,
there exists an $f$ such that the solution of the IIP with $g=0$ and
this particular $f$ satisfies $\abs{k}\|u\|_{\LtO}\gtrsim
\|f\|_{\LtO}$. Furthermore Lemma \ref{lem:ball} shows that if $\Omega$
is a ball and $f=0$ then there exists a $g$ such that the solution of
the IIP with $f=0$ and this particular $g$ satisfies
$\abs{k}\|u\|_{\LtO}\gtrsim \|g\|_{\LtG}$.

Note that Assumption \ref{ass:eta} includes the cases $\eta=\pm k$, and thus the bound \eqref{eq:thm3} holds 
for the two most-commonly occurring impedance boundary conditions, namely $\dnu - \ri k \gamma u=g$ and $\dnu + \ri k \gamma u =g$.

For our application of this result to integral equations, we state a result on the Dirichlet trace of the solution of the IIP.

\begin{corollary}[Bound on the interior impedance-to-Dirichlet map]\label{cor:ItD}
  Let $\Omega$ be a bounded $C^\infty$ domain in 2- or 3-d
  with boundary $\Gamma$.
Given $f\in \LtO$, $g\in \LtG$, and $\eta$ satisfying Assumption \ref{ass:eta}, let $u \in H^1(\Omega)$ be 
be the solution to
the interior impedance problem \eqref{eq:IIP}. Then
\beq\label{eq:ItD}
\N{\nT(\gamma u)}_{\LtG} + \abs{k} \N{\gamma u}_{\LtG} \lesssim \N{f}_{L^2(\Omega)}+\N{g}_{\LtG}
\eeq
for all $k \in \RR.$
If the stronger Assumption~\ref{ass:eta2} holds,
estimate \eqref{eq:ItD} holds with $1+\abs{k}$ replacing $\abs{k}.$
\end{corollary}

We now state two further corollaries, which are relevant for the numerical analysis of finite-element discretizations for the IIP. For simplicity, we state them for $|k|$ bounded away from zero.

\begin{corollary}[Bound on the inf-sup constant]\label{cor:infsup}
Let $\Omega$ be a bounded $C^\infty$ domain in 2- or 3-d
  with boundary $\Gamma$.
Given $f\in (H^1(\Omega))'$, $g\in H^{-1/2}(\Gamma)$, and $\eta$ satisfying Assumption \ref{ass:eta}, let $u \in H^1(\Omega)$ be the solution to
the interior impedance problem \eqref{eq:IIP}. Then, given $k_0>0$,
\beq\label{eq:infsup}
\N{\gu}_{\LtO} + \abs{k} \N{u}_{\LtO}
\lesssim|k| \left(
\N{f}_{(H^1(\Omega))'}+ \N{g}_{\HmhG}\right)
\eeq
for all $\abs{k}\geq k_0$. Furthermore,
\beq\label{eq:infsup2}
\inf\limits_{0\neq u\in H^1(\Omega)} \sup\limits_{0\neq v\in H^1(\Omega)}
\frac{|a(u,v)|}{\N{u}_{H_k^1(\Omega)} \N{v}_{H_k^1(\Omega)} } \gtrsim \frac{1}{\abs{k}},
\eeq
where $a(\cdot,\cdot)$, defined by \eqref{eq:sesqui} below, is the sesquilinear form of the variational formulation of the interior impedance problem, and $\|\cdot\|_{H_k^1(\Omega)}$ is the weighted $H^1$-norm defined by \eqref{eq:H1def} below.
\end{corollary}

\begin{corollary}[Bound on the $H^2$-norm]\label{cor:H2}
  Let $\Omega$ be a bounded $C^\infty$ domain in 2- or 3-d
  with boundary $\Gamma$.
Given $f\in \LtO$, $g\in \HhG$, and $\eta$ satisfying Assumption \ref{ass:eta}, let $u \in H^1(\Omega)$ be 
be the solution to
the interior impedance problem \eqref{eq:IIP}. Then, given $k_0>0$, 
\beq\label{eq:H2}
\N{u}_{H^2(\Omega)} \lesssim \abs{k} \left(
\N{f}_{\LtO}+ \N{g}_{\HhG}\right)
\eeq
for all $\abs{k}\geq k_0$.
\end{corollary}

Shifting to a slightly different perspective, having proved the bound
\eqref{eq:thm3} for real $k$ it is natural to impose the
\emph{homogeneous} impedance boundary condition $\dnu- \ri\eta \gamma
u=0$ and consider the resolvent-like operator family defined by solving the Helmholtz
equation with this ($k$-dependent!) boundary-condition.   That is, we define
$$
\Rimp(k):
L^2(\Omega)\rightarrow L^2(\Omega)$$ by
$$\Rimp(k) f=u,$$ where $u$ is the solution to $$(\Delta
+k^2)u=f$$ satisfying $$\dnu- \ri \eta \gamma u=0.$$ If $\eta$ satisfies
Assumption~\ref{ass:eta} then $\Rimp(k)$ is well defined when $k\in
\Rea\setminus\{0\}$.  Meanwhile, the strict positivity of $\notalpha$ implies
that $\Rimp(k)$ is well defined and
holomorphic for $\Im k> 0.$ We
can then use a simple perturbation argument to show the existence of
regions beneath the real
axis free of poles (the equivalent of ``resonances'' in this compact,
non-self-adjoint setting); if we strengthen our assumptions to strict
positivity of $\notbeta,$ this yields a full pole-free strip beneath the real
axis, while mere nonnegativity leaves the possibility of a singularity
at $k=0.$

The following result is stated with the stronger hypothesis and
consequent pole-free strip.

\bth[Pole-free strip beneath the real axis]\label{thm:pole}
The operator family $\Rimp(k):\LtO\rightarrow \LtO$ defined as the inverse
of $(\Lap+k^2)$ with boundary condition $\dnu - \ri \eta \gamma u=0$, where
$\eta$ satisfies Assumption \ref{ass:eta2}, is holomorphic on $\Im k
>0$. Furthermore there exist an $\eps>0$ such that $\Rimp(k)$
extends from the upper-half plane to a holomorphic operator family on $\Im
k> -\eps$, satisfying the uniform estimate \beq\label{eq:55}
\N{\Rimp(k)}_{\LtO \to \LtO} \lesssim (1+ |k|)^{-1} \eeq in that
region.
\end{theorem}

\subsection{Discussion of previous results related to Theorems \ref{thm:1}--\ref{thm:3} and \ref{thm:pole}, and high-frequency estimates for the Helmholtz equation in general}\label{sec:previous}

The main previously-existing sharp bound for one of the DtN and NtD maps is
the bound \eqref{eq:DtN2} proved when $\Oi$ is a Lipschitz domain that is
star-shaped with respect to a ball (in the sense of Part (ii) of Definition
\ref{def:star}). This bound was proved by Morawetz and Ludwig in \cite{MoLu:68} without the
smoothness requirements of the boundary explicitly stated, but the same
techniques apply to Lipschitz domains, modulo some additional technical
work; see \cite[Remark 3.8]{Sp:14} and \cite[Appendix A]{MoSp:14}.
The DtN bounds \eqref{eq:DtN1} and \eqref{eq:DtN2} were also obtained in the strictly convex case by Cardoso,
Popov, and Vodev in \cite{CaPoVo:01} as well as by Sj\"ostrand
\cite{Sj:14}; see also the parametrix construction in the appendix of
\cite{StVo:95}.
Non-sharp bounds on the DtN and NtD maps were proved in \cite{Ba:71},
\cite{LaVa:12}, and \cite{Sp:14}; see \cite[\S1.2]{Sp:14} for a discussion
of all these results.

Of the bounds on the IIP in the literature, the only previously-existing
sharp result was that \eqref{eq:thm3} holds when $\Omega$ is Lipschitz and
star-shaped with respect to a ball. This was proved in 2-d when $\Gamma$ is
piecewise smooth by Melenk \cite[Proposition 8.1.4]{Me:95} and in 3-d by
Cummings and Feng \cite[Theorem 1]{CuFe:06}. The technical work referred to
above can then be used to establish the bound when $\Gamma$ is Lipschitz
(see, e.g., \cite[Theorem 2.6]{GaGrSp:15} where the analogue of this bound
is proved for a more general class of wavenumbers). By the discussion
immediately after Theorem \ref{thm:3}, this bound for star-shaped Lipschitz
domains is sharp. Bounds for general Lipschitz domains with positive powers
of $k$ in front of both $\|f\|_{\LtO}$ and $\|g\|_{\LtG}$ were obtained in
\cite[Theorems 3.6 and 4.7]{FeSh:94}, \cite[Theorem 2.4]{EsMe:12}, and
\cite[Theorem 1.6]{Sp:14}; see \cite[\S1.2]{Sp:14} for more discussion.

Regarding the pole-free strip result of Theorem \ref{thm:pole},
  the analogous result for the exterior impedance problem follows from
  the exponential decay result of \cite{Al:02} for the wave equation
  with damped boundary conditions (in an analogous way to how Theorem
  \ref{thm:pole} followed from the exponential decay in
  \eqref{eq:BLR2}).
Furthermore, the recent work of \cite{Pe:15} on the exterior impedance problem
gives quite precise bounds on the locations of poles much deeper
in the lower half-space than those considered here.

A crucial ingredient in the estimates obtained in this paper is the \emph{nontrapping resolvent
    estimate,} which we use to solve away errors for both Dirichlet and Neumann
  exterior problems.  If $\Omega_+$ is nontrapping, we
  have for any $\chi \in C^\infty_c(\Omega_+)$
\begin{equation}\label{nontrappingresest}
\norm{\chi (\Lap+k^2)^{-1} \chi}_{L^2(\Omega_+) \to L^2(\Omega_+)}\leq
C (1+\smallabs{k})^{-1},\quad k \in \RR
\end{equation}
(see Theorem~\ref{thm:resolve} below for a slightly refined
formulation and generalizations).  This result follows from a
combination of two separate ingredients.  By work on propagation of
singularities for the wave equation on manifolds with boundary by
Melrose \cite{Me:75}, Taylor \cite{Ta:76}, and Melrose--Sj\"ostrand
\cite{MeSj:78}, we know that solutions to the wave equation on
nontrapping domains with compactly supported initial data become
smooth for $t \gg 1.$ A parametrix method of Vainberg \cite{Va:75} or
the methods of Lax--Phillips \cite{LaPh:89} can then be used to turn
this ``weak Huygens principle'' into a resolvent estimate (and indeed
to obtain a
region of analyticity below the real axis for the analytic
continuation of the cutoff resolvent).  The estimate
\eqref{nontrappingresest} is known to fail, by contrast, whenever
there are trapped orbits, by work of Ralston \cite{Ra:69}.  We mainly
use the estimate \eqref{nontrappingresest} as a black box in our
estimates below, but we do need to
return to the Vainberg parametrix construction to prove a
variant of \eqref{nontrappingresest} that deals with Dirichlet data for the
nontrapping Neumann resolvent (Lemma~\ref{lemma:neumannresolvent}.)

\subsection{The main ideas used to obtain Theorems \ref{thm:1}-\ref{thm:3} and \ref{thm:pole}}

We now give a brief overview of how the main results were obtained, with more detail naturally given in \S\ref{sec:DtN}-\ref{sec:IIP}.

In contrast to the proofs of the bounds on the NtD map and IIP, 
our proof of the DtN map bounds in Theorem \ref{thm:1} takes places solely in the setting of 
stationary
scattering theory, i.e., we never consider the associated problem for the wave equation. We use a ``gluing" argument, where 
outgoing solutions for the far-field are ``glued" to solutions of an ``auxiliary problem" in a bounded region. This type of argument goes back at least to Lax and Phillips \cite[\S5]{LaPh:73} and was used to obtain (non-sharp) bounds on the DtN map in \cite{LaVa:12} and \cite{Sp:14}. Our contribution is to choose a different auxiliary problem to that considered in \cite{LaVa:12} and \cite{Sp:14}, with this change then yielding the sharp result.

The main ingredient for our proof of the NtD map bounds in Theorem \ref{thm:2} is a collection of restriction bounds for solutions of the wave equation with Neumann boundary conditions due to Tataru \cite{Ta:98}. These are used in conjunction with the Vainberg parametrix construction briefly discussed in \S\ref{sec:previous} above.

For the bound on the IIP in Theorem \ref{thm:3} we use the results of Bardos, Lebeau, and Rauch \cite{BaLeRa:92} on exponential decay of the energy of solutions of the wave equation with damped boundary conditions, with the estimate \eqref{eq:thm3} obtained by a Fourier-transform argument. Once \eqref{eq:thm3} has been established for $k\in \mathbb{R}$, the pole-free strip result in Theorem \ref{thm:pole} then follows by a standard perturbation argument.

\subsection{Application of the above results to integral equations}\label{sec:1-3}

As mentioned above, the results of Theorems \ref{thm:1}, \ref{thm:2}, and \ref{thm:3} can be applied to integral equations. Our main result in this direction concerns the standard integral equation used to solve the Helmholtz exterior Dirichlet problem.

When $u$ is the solution to the Helmholtz exterior Dirichlet problem, the Neumann trace of $u$, $\dnpu$, satisfies the integral equation
\beq\label{eq:CFIE}
\opA (\dnpu) = f_{k,\eta}
\eeq
on $\Gamma$, where the integral operator $\opA$ is the so-called \emph{combined-potential} or \emph{combined-field} integral operator (defined by \eqref{eq:CFIEdef} below), 
$f_{k,\eta}$ is given in terms of the known Dirichlet data $\gamma_+ u$ (see \eqref{eq:CFIE2}).
Usually the parameter $\eta$ is a real constant different from zero, but in
fact $\eta$ will also be allowed to be a function of position on $\Gamma.$

We introduce the notation that $\DtN$ denotes the exterior DtN map, as a mapping from $H^{s+1/2}(\Gamma) \rightarrow H^{s-1/2}(\Gamma)$ for $|s|\leq 1/2$, and $\ItD$ denotes the interior impedance-to-Dirichlet map, as a mapping from $H^{s-1/2}(\Gamma) \rightarrow H^{s-1/2}(\Gamma)$ for $|s|\leq 1/2$ (see \S\ref{sec:2-1} below and \cite[Theorems 2.31 and 2.32]{ChGrLaSp:12} for details on how these maps are defined for these ranges of spaces).

The inverse of $\opA$ can be written in terms of the exterior DtN map $\DtN$ and interior impedance to Dirichlet map $\ItD$ as follows 
\beq\label{eq:key}
\opAinv = I - (\DtN - \ri \eta ) \ItD;
\eeq
this decomposition is implicit in much of the work on the
combined-potential operator $\opA$, but (to the authors' knowledge)
was first written down explicitly in \cite[Theorem
2.33]{ChGrLaSp:12}. We give another, more intuitive, proof of this
result in Lemma~\ref{lemma:inverses} below. 

The operator $\opA$ is usually considered as a operator from $\LtG$ to itself (the reasons for this are explained in \S\ref{sec:int}) and the bounds on the exterior DtN map and interior impedance-to-Dirichlet map in Theorem \ref{thm:1} and Corollary \ref{cor:ItD} immediately yield the following bound on $\|\opAinv\|_{\LtGt}$.

\begin{theorem}\label{thm:CFIE}
Let $\Oe\subset \Rea^d$, $d=2,3$, be a nontrapping domain and suppose that $\eta$ satisfies Assumption \ref{ass:eta}. 
Then, given $k_0>0$,
\beq\label{eq:Ainv_bound_main}
\normAinv_{\LtGt} \lesssim 1
\eeq
for all $\abs{k}\geq k_0$.
\end{theorem}

Since the proof is so short, we include it in this introduction.  The
spaces $\HokG$ used below are weighted Sobolev spaces defined in
\S\ref{sec:notation} (in particular, see equation~\eqref{eq:H1def}).
\bpf
The decomposition \eqref{eq:key} implies that
\begin{multline}
\normAinv_{\LtGt} \leq 1 + \N{\DtN}_{\HokG\rightarrow\LtG} \N{\ItD}_{\LtG\rightarrow\HokG}\\ + |\eta| \N{\ItD}_{\LtGt}.
\end{multline}
Theorem \ref{thm:1} implies that $\N{\DtN}_{\HokG\rightarrow\LtG}\lesssim 1$ and Corollary \ref{cor:ItD} implies that 
$\N{\ItD}_{\LtG\rightarrow\HokG}\lesssim 1$ (and thus $\N{\ItD}_{\LtG\rightarrow\LtG}\lesssim \abs{k}^{-1}$). These results, along with the assumption on $\eta$, immediately give \eqref{eq:Ainv_bound_main}.
\epf

We make two immediate remarks regarding Theorem \ref{thm:CFIE}.
\ben
\item The bound \eqref{eq:Ainv_bound_main} is sharp, since it was proved in \cite[Theorem 4.3]{ChGrLaLi:09} that $\normAinv_{\LtG} \geq 2$ when part of $\Gamma$ is $C^1$ and $d=2,3$. 
\item In this paper we focus on the \emph{direct} integral equation for the
  exterior Dirichlet problem, i.e., the equation where the unknown has an
  immediate physical meaning (in this case, it is the Neuman trace $\dnpu$)
  but an analogous bound to \eqref{eq:Ainv_bound_main} holds for the
  inverse of the operator involved in the standard \emph{indirect} integral
  equation (where the unknown of the integral equation does not have an
  immediate physical meaning); see, e.g., \cite[Remark 2.24,
  \S2.6]{ChGrLaSp:12}.  
\een

There have been two previous upper bounds on $\normAinv_{\LtGt}$ proved in the literature; the bound  
\beq\label{eq:Ainv_CWM}
\normAinv_{\LtGt} \lesssim 1 + \frac{k}{|\eta|}
\eeq
when $\Oi$ is a 2- or 3-d Lipschitz domain that is star-shaped with respect to a ball and $\eta\in \Rea\setminus\{0\}$
was proved in \cite[Theorem 4.3]{ChMo:08} using the Morawetz-Ludwig DtN bound and Melenk's bound on the IIP, both discussed in 
\S\ref{sec:previous}.
Furthermore, using non-sharp bounds on $\DtN$ and $\ItD$, the bound
\beq
\label{eq:Ainv_Sp}
\normAinv_{\LtGt} \lesssim k^{5/4}\left(1 + \frac{k^{3/4}}{|\eta|}\right)
\eeq
for $\eta \in \Rea\setminus\{0\}$ was proved in \cite[Theorem 1.11]{Sp:14} when either $\Oi$ is a 2- or 3-d nontrapping domain, or $\Oi$ is a nontrapping polygon.

An immediate application of the bound \eqref{eq:Ainv_bound_main} is the following. An error analysis of the $h$-boundary element method (i.e.~the Galerkin method using subspaces consisting of piecewise polynomials with fixed degree) applied to the equation \eqref{eq:CFIE} was conducted in \cite{GrLoMeSp:15}. This analysis required $\normAinv_{\LtGt}\lesssim 1$, and so covered the case when $|\eta|\sim k$ and $\Oi$ is star-shaped with respect to a ball, using the bound \eqref{eq:Ainv_CWM}. Thanks to the bound \eqref{eq:Ainv_bound_main}, however, this analysis is now valid when $\Oe$ is nontrapping and $\eta$ satisfies Assumption \ref{ass:eta}. (Note that the error analysis of the $hp$-boundary element method conducted in \cite{LoMe:11}, \cite{Me:12} only requires $\normAinv_{\LtGt}\lesssim k^\beta$ for some $\beta>0$, and thus the bound \eqref{eq:Ainv_Sp} is sufficient for this analysis to be valid for nontrapping domains.) 

The bound \eqref{eq:Ainv_bound_main}, used in conjunction with the recent
results of Gal\-kowski--Smith and Gal\-kowski \cite{GaSm:14}, \cite{HaTa:14},
on essentially the norm of $\opA$, almost completes the study of the
conditioning of $\opA$ in the high-frequency limit, i.e., the study of
\beq\label{eq:cond_def} \cond(\opA) := \normA_{\LtGt} \normAinv_{\LtGt}
\eeq for $k$ large.  This study was initiated back in the 80s for the case
when $\Oi$ is a ball \cite{Kr:85}, \cite{KrSp:83}, \cite{Am:90}, with the
main question considered being how one should choose the parameter $\eta$
to minimize the condition number. The first works to consider domains
other than balls 
were \cite{ChGrLaLi:09}, \cite{ChMo:08}.  We discuss the
implications of Theorem \ref{thm:CFIE} and \cite{HaTa:14} on the condition
number of $\opA$ and the choice of $\eta$ in \S\ref{sec:condition}.

So far we have only discussed integral equations for the exterior Dirichlet problem. The case of the exterior Neumann problem is more subtle, and we refer the reader to \S\ref{section:Neumann}--\S\ref{sec:6-3} where this is discussed.

This subsection has discussed the application of the bounds of
Theorems \ref{thm:1}--\ref{thm:3} to boundary integral equations for
real $k$. 

In a different direction, the pole-free strip
for the IIP in Theorem \ref{thm:pole} has the following two applications in the theory of boundary integral equations. 
\ben
\item This result is used in \cite{EpGrHa:15}, along with results from
classical scattering theory and effectively the relation
\eqref{eq:key}, to show that $\opA$ is
invertible for $\Im k> -\delta$, for some $\delta>0$, when $\Oe$ is
nontrapping and $\eta$ satisfies Assumption \ref{ass:eta2}. 
\item The method of \cite{ZhBa:15} for finding Dirichlet eigenvalues of the Laplacian using boundary integral equations relies on the existence of a pole-free strip for both the interior and exterior impedance problems (see \cite[Remark 7.5]{ZhBa:15}). The former is guaranteed by Theorem \ref{thm:pole}, and the latter is guaranteed by \cite{Al:02}.
\een


\section{Notation and preliminaries}\label{sec:notation}
Let $\Oi\subset \Rea^d, \,d\geq 2,$ be a bounded, Lipschitz open set with boundary $\Gamma:=\partial \Oi$,
such that the open complement $\Oe:= \Rea^d \setminus \overline{\Oi}$ is connected.
We denote the exterior and interior traces by $\gamma_\pm$, and the exterior and interior normal-derivative traces by  $\partial_n^{\pm}$.
The symbol $\chi$ will denote a function in $C^\infty_{c}(\Oe)$
that equals one in a neighborhood of $\Oi$.  Additional assumptions about the support of particular cutoffs will be stated explicitly.

The symbol $\Delta$ denotes the (nonpositive) Laplacian and $\Box$
denotes the wave operator $\pd[t]^{2} - \Delta$.

Given a function $u\in C^1(\Rea^d\setminus \overline{B_{R_0}})$ for some $R_0>0$ and given $\lambda \in \CC,$ we say that
$u$ satisfies the Sommerfeld radiation condition with spectral parameter $\lambda$ if
\beq\label{eq:src2}
\pdiff{u}{r}
- \ri \lambda u
= o\left(\frac{1}{r^{(d-1)/2}}\right) 
\eeq
as $r:=|\bx| \tendi$, uniformly in $\hat{\bx}:=\bx/r$.

We define the weighted norm
\begin{equation}\label{eq:H1def}
\norm{u}^2_{H^1_k(X)} :=\norm{\nabla u}^2_{L^2(X)}+k^2\norm{u}^2_{L^2(X)}.
\end{equation}
(we use this notation with $X$ either $\Oe,$ $\Oi,$ or $\Gamma$; in the latter case the gradient is to be understood as the surface gradient $\nT$).

More generally, for $s\in \RR$ we let $H^{s}_k(X)$ denote the weighted Sobolev space obtained by interpolation
and duality from the spaces of positive integer order
$$
H^m_k(X)  = \big\{u\in L^2(X): \abs{k}^{m-\smallabs{\alpha}} D^\alpha u \in
L^2(X),\ \text{for all } \abs{\alpha} \leq m\big\}.
$$
As usual (see e.g.\ \cite[\S4.4]{Ta:96}) we may identify these spaces
on manifolds with boundary with the
quotient space
$$
H^s_k(\Omega_\pm) = \big\{u \in H^s_k(\RR^n)\big\}/\big\{u: u\rvert_{\Omega_\pm}=0\big\}.
$$

An easy interpolation (see, e.g., \cite{ChHeMo:14}) shows that an equivalent norm on $H^s_k(X)$ for $s>0$
is $\norm{\bullet}_{H^s}+ \abs{k}^s \norm{\bullet}_{L^2},$ and we will use this
fact freely below.

We will also have occasion to consider domain of the self-adjoint operator
$(-\Lap+k^2)^{s/2},$ with $\Lap$ denoting the (nonpositive) Laplacian with
Neumann or Dirichlet boundary conditions and $s\geq 0.$ We let
$\dom^s_{N,k}$ resp.\ $\dom^s_{D,k}$ denote these respective domains; for
negative $s$ the spaces are defined by duality: $\dom^s_{\bullet,k}=(\dom^{-s}_{\bullet,k})^*.$ As in
\cite[\S5.A]{Ta:96}, we note that $\dom^1_N(\Omega_\pm)=H^1_k(\Omega)$
and so by interpolation we have
\begin{equation}\label{domaintosobolev}
H^s_k(\Omega_\pm) = \dom^s_N(\Omega_\pm),\ s \in [0,1].
\end{equation}

The norm with no subscript attached, $\smallnorm{\bullet},$
will denote the $L^2$ norm throughout.

The following lemma connects Sobolev regularity in space-time to weighted
Sobolev regularity following Fourier transform.  Let $\mathcal{F}^{-1}$ denote
the inverse Fourier transform taking the time variable to frequency variable $k.$
\begin{lemma}\label{lemma:FTSobolev}
Let $I\subset \RR$ be a bounded open interval.
There exist $C_I$ such that 
$$
\norm{\mathcal{F}_{t\to k}^{-1} u(k,x) }_{H^\alpha_k(X)} \leq C_I
\norm{u}_{H^\alpha(I\times X)}
$$
for every $u \in H^\alpha(\RR\times X)$ supported in $I \times X.$
\end{lemma}
The proof is simply intertwining the elliptic operator
$(\pa_t^2+\Lap)$ with the Fourier transform to obtain the result for
$\alpha \in \NN,$ followed by interpolation and duality for the
general case.

\subsection{Preparatory results for proving Theorems \ref{thm:1} and \ref{thm:2} (the DtN and NtD bounds)}\label{sec:2-1}

The following interpolation result (which appears as \cite[Lemma 2.3]{Sp:14}) shows that the DtN bound \eqref{eq:DtN1} follows from \eqref{eq:DtN2}, and the NtD bound \eqref{eq:NtD1} follows from \eqref{eq:NtD2}.
To state this result, we denote the DtN map in $\Oe$ by $\DtN$ and the NtD map by $\NtD$ (following the notation in \cite[\S2.7]{ChGrLaSp:12}). $\DtN$ is defined as a map from $\HhG$ to $\HmhG$ by standard results about the solvability of the exterior Dirichlet problem and the definition of the normal derivative, and the regularity result of Ne\v{c}as stated as Lemma \ref{lem:Necas} below implies that $\DtN$ can be extended to a map from $\HoG$ to $\LtG$.
Analogous arguments hold for $\NtD$.

\ble\textbf{\emph{(}\cite[Lemma 2.3]{Sp:14}\emph{)}}\label{lem:inter}
With $\Oe$, $\DtN$, and $\NtD$ defined above,
\beqs
\N{\DtN}_{\HhG \rightarrow \HmhG} \leq \N{\DtN}_{\HoG \rightarrow \LtG}
\eeqs
and analogously,
\beqs
\N{\NtD}_{\HmhG \rightarrow \HhG} \leq \N{\NtD}_{\LtG \rightarrow \HoG}.
\eeqs
\ele

(Note that an analogous result holds for the interior impedance-to-Dirichlet map, and thus the bound in Corollary \ref{cor:ItD} implies a bound on this map from $\HmhG$ to $\HhG$, but we do not need this latter result in this paper.)

Having reduced the problem of obtaining the DtN and NtD bounds in
Theorems \ref{thm:1} and \ref{thm:2} to the problem of obtained the
bounds between the spaces $\HoG$ and $\LtG$, we now use the well-known
fact that a Rellich-type identity can be used to bound the (highest
order terms of the) DtN and NtD maps, modulo
terms in the domain. The next lemma is a restatement of Ne\v{c}as'
result for strongly elliptic systems (see \cite[\S5.1.2,
5.2.1]{Ne:67}, \cite[Theorem 4.24]{Mc:00}) applied to the specific
case of the Helmholtz equation, where we have kept track of the
dependence of each term on $k$ (see \cite[Lemma 3.5]{Sp:14} for
details).

\ble[DtN and NtD bounds in $H^1(\Gamma)$--$\LtG$ modulo terms in the domain]
\label{lem:Necas}
With $\Oe$ and $\chi$ as above, given $f\in L^2_{\rm comp}(\Oe)$, let $u \in \Holoce(\Oe)$ be a solution to $\Delta u +k^2 u =-f$.

(i) If $\gamma_+ u \in \Ho{\Gamma}$ then $\dnpu \in \Lt{\Gamma}$ and
\beq\label{eq:41}
\N{\dnpu}^2_{\LtG} \lesssim \N{\nT (\gamma_+ u)}^2_{\LtG} + \N{\chi u}^2_{\HokOe}+ \N{f}^2_{\Lt{\Oe}}.
\eeq

(ii) If $\dnpu \in \Lt{\Gamma}$ then $\gamma_+ u \in \Ho{\Gamma}$ and
\beq\label{eq:42}
\N{\nT (\gamma_+ u)}^2_{\LtG}  \lesssim \N{\dnpu}^2_{\LtG} + \abs{k}^2 \N{\gamma_+ u}^2_{\LtG}
+ \N{\chi u}^2_{\HokOe}+ \N{f}^2_{\Lt{\Oe}}.
\eeq
\ele

Therefore, to prove the bounds in Theorem \ref{thm:1} it is sufficient to prove that, if $u\in \Holoc(\Oe)$ is the solution to the exterior Dirichlet problem for the homogeneous Helmholtz equation, with $H^1$-Dirichlet boundary data $g_D$, then 
\beqs
\N{\chi u}_{\HokOe} \lesssim \N{\gamma_+ u}_{\HokG}.
\eeqs
Similarly, to prove the bounds in Theorem \ref{thm:2} it is sufficient to prove that, 
$u\in \Holoc(\Oe)$ is the solution to the exterior Neumann problem for
the homogeneous Helmholtz equation, with $L^2$-Neumann boundary data
$g_N$, then with $\beta$ as in Theorem~\ref{thm:2},
\begin{align*}
\N{\gamma_+ u}_{\LtG} \lesssim k^{-\beta}\N{\dnpu}_{\LtG} \quad\tand\quad
\N{\chi u}_{\HokOe} \lesssim k^{1-\beta}\N{\dnpu}_{\LtG}
\end{align*}
(we will actually prove the stronger result that the second bound holds with a smaller power of $k$ on the RHS, but this will not affect the bound on the NtD map).
The asymmetry between what we need to prove for the Neumann problem versus what we need to prove for the Dirichlet problem is due to the fact that only the $H^1$-semi norm of the Dirichlet trace is controlled in \eqref{eq:42}, which is due to the structure of the Rellich identity (see, e.g., \cite[Equation 3.13]{Sp:14}).

Finally, in our proof of the NtD estimates we will need the following lemma. It is perhaps easiest to state this in terms of norms of $u$ over 
$\OR := \Oe \cap B_R$, where $B_R:=\{\bx: |\bx|< R\}$, but the result could be translated into norms of $\chi u$ over $\Oe$ for appropriate cut-off functions $\chi$.

\begin{lemma}[Bounding the $H^1$ norm via the $L^2$ norm and the data]\label{lem:H1L2}
Given $f\in L^2_{\comp}(\Oe)$, let $u \in \Holoce(\Oe)$ be a solution of 
the Helmholtz equation $\Delta u +k^2 u= -f$ in $\Oe$. Then, given
$R>\sup_{x\in\Oi}|x|$,
\beqs
\N{\gu}_{\Lt{\OR}}^2 \lesssim \ang{k}^2 \N{u}^2_{\Lt{\Omega_{R+1}}} + \ang{k}^{-2}\N{f}^2_{\Lt{\Oe}}+ \N{\gamma_+ u}_{\LtG} \N{\dnpu}_{\LtG}
\eeqs
for all $k\in\Rea$.  
\ele

This result when 
one of $\gamma u$ and $\dnu$ is zero is proved in \cite[Lemma
2.2]{Sp:14}; a similar result appears in \cite[Lemma 1]{Mo:75}.

\section{Exterior Dirichlet-to-Neumann estimates}\label{sec:DtN} 

In this section we prove Theorem \ref{thm:1}, i.e., a bound on the exterior Dirichlet-to-Neumann map for solutions of the
Helmholtz equation satisfying the Sommerfeld radiation condition.  

The methods used here will be completely in the setting of stationary
scattering theory, i.e., we will never have recourse to energy estimates
for solutions to the wave equation (which is, of course, connected via
Fourier transform).  The energy estimates that we present are more widely
known in this latter setting, however---cf.\ H\"ormander
\cite[\S24.1]{Ho:85} as well as the more general estimates of Kreiss and
Sakamoto in the context of general hyperbolic systems with a boundary
condition satisfying the uniform Lopatinski condition \cite{Kr:70},
\cite{Sa:70}, \cite{Sa:70a}. (In contrast, when dealing with the Neumann-to-Dirichlet
operator below, we need to use results known only in the wave equation
setting.)  

More specifically, the method we use to prove Theorem \ref{thm:1} consists of a ``gluing" argument, where 
outgoing solutions for the far-field are ``glued" to solutions of an ``auxiliary problem" in a bounded region; this type of argument goes back at least to Lax and Phillips \cite[\S5]{LaPh:73}. 
In our situation, estimates for the DtN map for a lower-order ``perturbation" of the Helmholtz equation are used in conjunction with the resolvent estimate for the problem with homogeneous boundary conditions. This argument was first used to obtain bounds on the DtN map in \cite{LaVa:12}, and later refined in \cite{Sp:14}. Both these previous works use the equation $\Delta w -k^2 w=0$ as the lower-order perturbation, and obtain non-sharp bounds on the Helmholtz DtN map. Here we use the equation $\Delta w + (k^2 + \ri |k|)w=0$ as the lower-order perturbation (i.e., the Helmholtz equation with some absorption/damping), and this change is sufficient to prove the sharp result.


Before we begin, it is helpful to recall the following resolvent estimates for the Dirichlet problem
(all but one of which hold for the Neumann problem as well).

\begin{theorem}[Resolvent estimates]\label{thm:resolve}
Let $f\in \Lt{\Oe}$ have compact support, and let $u\in \Holoce(\Oe)$ 
be a solution to the Helmholtz equation $\Delta u +k^2 u= -f$ in $\Oe$ that satisfies the Sommerfeld radiation condition \eqref{eq:src} (with $\la=k$) and the boundary condition $\gamma_+ u=0$.
If \emph{either}
\bit
\item[(a)] $\Oe$ is a 2- or 3-d nontrapping domain (in the sense of Definition \ref{def:nt1}) \emph{or} 
\item[(b)] $\Oi$ is a nontrapping polygon (in the sense of Definition \ref{def:nt2}), \emph{or}
\item[(c)] $\Oi$ is a 2- or 3-d Lipschitz domain that is star-shaped (in the sense of Definition \ref{def:star}(i)) 
\eit
then, given $k_0>0$, 
\beq\label{eq:resolvent}
\N{\chi u}_{\HokOe}
\lesssim \N{f}_{\Lt{\Oe}}
\eeq
for all $\abs{k}\geq k_0$.
\end{theorem}

\begin{proof}
The result for Part (a) is proved in \cite[Theorem
7]{Va:75} using the
propagation of singularities results of \cite{MeSj:78},
\cite{MeSj:82}.   (See also Vainberg's book \cite{Va:89} for a broader
survey of these methods.)
The result for Part (b) was
proved when $\Oi$ is a nontrapping polygon in \cite[Corollary 3]{BaWu:13}.
The bound \eqref{eq:resolvent} was proved when $\Oi$ is a star-shaped domain in 2- or 3-d in \cite[Lemma 3.8]{ChMo:08}. 
\epf

\begin{lemma}
If $w$ satisfies
\beq\label{eq:PDEw}
\Delta w +(k^2 + \ri \abs{k}) w = 0 \quad\text{ in }\Oe
\eeq
and the Sommerfeld radiation condition \eqref{eq:src} with spectral parameter $\sqrt{k^2+\ri \abs{k}}$, 
then, given $k_0>0$, 
\beq\label{eq:1}
\N{w}_\HokOe^2\lesssim \abs{k} \N{\gamma_+ w}_{\LtG} \N{\dnpw}_{\LtG}.
\eeq
for all $\abs{k}\geq k_0$.
\end{lemma}

\bpf
Given $k_0>0$, there exists a $c>0$ such that $\Im \sqrt{k^2 + \ri \abs{k} } \geq c$; therefore, since $w$ satisfies the Sommerfeld radiation condition and the associated asymptotic expansion (see, e.g., \cite[Theorem 3.6]{CoKr:83}), $w$ decays exponentially at infinity; hence both $w$ and $\nabla w$ are both in $L^2(\Oe)$.

We can therefore apply Green's identity (i.e., multiply the PDE \eqref{eq:PDEw} by $\overline{w}$ and integrate by parts), and obtain that
\beqs
-\int_\Gamma \overline{\gamma_ +w}\,\dnpw + \int_\Oe (k^2 + \ri \abs{k}) |w|^2 -|\nabla w|^2 =0.
\eeqs
Taking the imaginary part of this last expression and using the Cauchy-Schwarz inequality yields
\beq\label{eq:A}
\abs{k}\N{w}^2_{L^2(\Oe)} \leq \N{\gamma_+ w}_{\LtG} \N{\dnpw}_{\LtG}.
\eeq
Taking the real part yields 
\beq\label{eq:B}
\N{\nabla w}_{L^2(\Oe)}^2 \leq k^2 \N{w}^2_{L^2(\Oe)} + \N{\gamma_+ w}_{\LtG} \N{\dnpw}_{\LtG},
\eeq
and combining \eqref{eq:A} and \eqref{eq:B} yields the result \eqref{eq:1}.
\epf

\ble[Bound on the exterior Dirichlet problem with damping]\label{lem:damp}

\

\noi Given $g_D \in H^1(\Gamma)$, let $w$ be the solution of
\beqs
\Delta w +(k^2 + \ri \abs{k}) w = 0 \quad\text{ in }\Oe, \quad\gamma_+ w =g_D \quad\ton \Gamma,
\eeqs
satisfying the Sommerfeld radiation condition \eqref{eq:src} (note
that the existence of a unique solution to this problem follows from
Remark \ref{rem:damp} below). Then
\beq\label{eq:Dbound_damp}
\N{w}_{\HokOe} \lesssim \N{g_D }_{\HokG}.
\eeq
\ele

\bre[Existence of outgoing solutions to the Dirichlet problem with damping]\label{rem:damp}

\

\noi If $w$ satisfies $\Delta w +(k^2 + \ri \abs{k}) w = 0$, then $w$ satisfies the Helmholtz equation $\Delta w +\la^2 w=0$ with $\la = \sqrt{k^2+\ri \abs{k}}$. Since $\Im \la>0$, 
the existence of outgoing solutions (i.e. solutions satisfying the
Sommerfeld radiation condition \eqref{eq:src}) to the Dirichlet and
Neumann problems for this equation follows in the same way as in the
case $\Im \la =0$. Indeed uniqueness is proved for $\Im \la \geq 0$ in
\cite[Theorem 3.13]{CoKr:83}. Existence in the case $\Im \la=0$ is proved using integral equation results in \cite[Corollary 2.28]{ChGrLaSp:12} (see also \cite[Theorem 2.10]{ChGrLaSp:12}), but the proof goes through in the exactly the same way when $\Im\la>0$.
\ere

\bpf[Proof of Lemma \ref{lem:damp}]
(i) Using the bound \eqref{eq:1} in the Ne\v{c}as result \eqref{eq:41} (with $w=u, f= \ri |k| w$) we find that
\beqs
\N{\dnpw}^2_{\LtG} \lesssim \N{\nT (\gamma_+ w)}^2_{\LtG} + \abs{k} \N{\gamma_+ w}_{\LtG} \N{\dnpw}_{\LtG}.
\eeqs
and so, absorbing the Neumann data term on the LHS
we have
\beqs
\N{\dnpw}_{\LtG} \lesssim \N{\nT (\gamma_+ w)}_{\LtG} + \abs{k} \N{\gamma_+ w}_{\LtG}. 
\eeqs
Using this last expression in \eqref{eq:1}, we obtain \eqref{eq:Dbound_damp}.
\epf

\begin{theorem}[Bounds on solutions of the Helmholtz Dirichlet problem]\label{thm:damp}

\

\noi Given $g_D \in H^1(\Gamma)$, let $u$ be the solution of
\beq\label{eq:Dprob}
\Delta u +k^2 u = 0 \quad\text{ in }\Oe, \quad\gamma_+ u =g_D,
\eeq
satisfying the Sommerfeld radiation condition \eqref{eq:src} (with $\la=k$). 
If $\Oe$ satisfies one of the conditions (a), (b), and (c) in Theorem \ref{thm:resolve} then
\beq\label{eq:Dbound}
\N{\chi u}_{\HokOe} \lesssim \N{g_D}_{\HokG}.
\eeq
\end{theorem}

\bpf
Let $w$ be as in Lemma \ref{lem:damp}.
Let $\chi\in C^\infty_{c}(\Oe)$ be equal to one in a neighborhood of $\Oi$, and define $v$ by $v:= u - \chi w$. This definition implies that $v\in \Holoce(\Oe)$ and satisfies the Sommerfeld radiation condition \eqref{eq:src} (with $\la=k$),
\beqs
\Delta v +k^2 v=h, \quad\text{and}\quad\gamma_+ v=0,
\eeqs
where 
\beqs
h:= \ri \abs{k} \chi w - w \Delta \chi - 2 \nabla w \cdot \nabla \chi.
\eeqs
Since $h$ has compact support, the resolvent estimate \eqref{eq:resolvent} implies that 
\beqs
\N{\chi v}_{\HokOe}\lesssim \N{w}_{\HokOe},
\eeqs
and thus
\beqs
\N{\chi u}_{\HokOe}\lesssim \N{w}_{\HokOe}.
\eeqs
Using the bound \eqref{eq:Dbound_damp}, we obtain the result \eqref{eq:Dbound}.


\epf

\begin{corollary}\label{cor:DtN}
If $\Oe$ satisfies one of the conditions (a), (b), and (c) in Theorem \ref{thm:resolve} and $u$ is the outgoing solution to the Dirichlet problem \eqref{eq:Dprob} then
\beq\label{eq:DtN}
\N{\dnpu}_{\LtG}\lesssim \N{g_D}_{\HokG}.
\eeq
\end{corollary}

\bpf
This follows from combining the bound \eqref{eq:Dbound} with Lemma \ref{lem:Necas}.
\epf

\bpf[Proof of Theorem \ref{thm:1}]
The bound \eqref{eq:DtN2} is proved in Corollary \ref{cor:DtN} above. The bound \eqref{eq:DtN1} then follows by Lemma \ref{lem:inter}.
\epf

\section{Exterior Neumann-to-Dirichlet estimates}\label{sec:NtD}

In this section we prove Theorem \ref{thm:2}, i.e. a bound on the exterior Neumann-to-Dirichlet map for solutions of the
Helmholtz equation satisfying the Sommerfeld radiation condition.  

This problem is subtler than obtaining bounds on the
Dirichlet-to-Neumann map, since the Neumann boundary condition does
not satisfy the uniform Lopatinski condition, hence the classic
estimates of Kreiss and Sakamoto do not apply to the wave equation,
nor does the simple stationary argument used above for the Dirichlet
problem.  Indeed, the problem becomes an intrinsically microlocal
one, with the degeneracy of the normal derivative at the glancing
set making even global energy estimates extremely sensitive to the
boundary geometry (which was irrelevant to energy estimates in the
Dirichlet case).

The main technical ingredient in our argument is a collection of estimates proved by Tataru \cite{Ta:98} for solutions to the wave equation with Neumann (or indeed many other) boundary conditions, which we now recall.  
The following is a restatement of part of Theorem~9 of
\cite{Ta:98}.

\begin{theorem}[Tataru]\label{theorem:tataru}
Let $\Gamma$ be smooth.
Suppose $v$ satisfies
\begin{equation}\begin{aligned}
\Box v &=0 \text{ on } \Omega_+ \times [0,T],\\
\partial_n^+ v&=g,\\
v(0)&=v_t(0)=0.
\end{aligned}
\end{equation}

Assume $g \in L^2(\Gamma\times [0,T])$.  Then 
$$v \in H^{\alpha}(\Omega_+\times [0,T])$$ and $$\gamma_+ v \in H^{\beta}(\Gamma \times [0,T]),$$
where\footnote{The positive curvature used here in dimensions $d=2,3$
  generalizes to be positive second
      fundamental form, in general dimension.} \begin{equation}\label{eq:alphabeta}
\begin{cases} \alpha=2/3,\ \beta=1/3 & \text{in general},\\
  \alpha=5/6,\ \beta=2/3 & \text{if } \Gamma \text{ has strictly
    positive curvature.}\end{cases}\end{equation}
\end{theorem}

Other results from \cite{Ta:98} that we shall use
(Theorems~3,5) estimate Dirichlet
data for solutions of the Helmholtz equation with homogeneous Neumann condition
and interior inhomogeneity:
\begin{theorem}[Tataru]\label{theorem:tataru2}
Let $\Gamma$ be smooth.
Suppose $v\in H^1_\loc$ satisfies
\begin{equation}\begin{aligned}
\Box v &=F \text{ on } \Omega_+ \times [0,T],\\
\partial_n^+ v&=0,\\
v(0)&=v_t(0)=0.
\end{aligned}
\end{equation}

Assume $F \in L^2(\Omega_+ \times [0,T]).$   Then
$$\gamma_+ v \in H^{\alpha}(\Gamma \times [0,T]),$$
where $\alpha$ is given by \eqref{eq:alphabeta}.
\end{theorem}

We now turn to an estimate analogous to the usual nontrapping
resolvent estimate that will allow us to estimate the Dirichlet data
of the Neumann resolvent for a nontrapping obstacle.

\begin{lemma}\label{lemma:neumannresolvent}
Assume that $\Omega_+$ is nontrapping.
Let $R_N(k)$ denote the outgoing Neumann resolvent on
$\Omega_+,$ acting on $f\in \dom^s_{N,k}.$  Then for $k\gg 1,$ for every $s\in \RR$
\begin{equation}\label{neumann1}
\norm{\chi R_N(k) \chi f}_{\dom^{s+1}_{N,k}} \lesssim \norm{f}_{\dom^s_{N,k}}
\end{equation}
and for $s \in [0,1]$
\begin{equation}\label{neumann2}
\norm{\gamma_+ R_N(k) \chi f}_{H^{s+\alpha}_k} \lesssim \norm{f}_{\dom^s_{N,k}}
\end{equation}
where $\alpha$ is given by \eqref{eq:alphabeta}.
\end{lemma}
We remark that $2\alpha=1+\beta.$

\begin{proof}
  The first part of this estimate is essentially the standard
  nontrapping resolvent estimate, albeit considered in more general
  weighted spaces than $L^2.$ The second part by contrast requires
  Tataru's boundary estimates together with an examination of the
  details of the Vainberg construction of a parametrix for the nontrapping resolvent
  \cite[Chapter X]{Va:89}.  This parametrix is indeed one of the usual routes to
  obtaining the standard resolvent estimate (\eqref{neumann1} with
  $s=0$) from the weak Huygens principle (eventual
  escape of singularities), and depends crucially on propagation of
  singularities results that enable us to conclude weak Huygens from
  nontrapping of billiard trajectories.
For details, we refer the reader to Theorem 2 in \cite{Va:89}, Chapter X; see also
\cite{Me:79} and \cite{MeSj:82} for the geometry and microlocal
analysis aspects.

To establish the first part of the result, we recall that Vainberg's
estimate (see also the ``black-box'' presentation of Vainberg's
method in \cite{TaZw:00}) yields
\begin{equation}\label{resest:1}
\norm{\chi_1 R_{N}(k) \chi_2}_{L^2\to L^2} \lesssim \smallang{k}^{-1}.
\end{equation}
We must extend to more general spaces in the domain and range.  
First, note that if $(\Lap+k^2) u=-f$ and $f$ has compact support in a fixed
region and $u$ satisfies the radiation condition then we of course can write, for $\chi_0$ compactly supported,
$$
\norm{\chi_0 \Lap u} \leq k^2 \norm{\chi_0 u}+\norm{\chi_0 f} \lesssim \ang{k} \norm{f},
$$
hence for any $\chi$ with smaller support than $\chi_0,$
\begin{equation}
\norm{\chi u}_{\dom^2_{N,k}} \lesssim \smallang{k}\norm{f},
\end{equation}
i.e., in particular 
\begin{equation}\label{resest:2}
\norm{\chi R_{N}(k) \chi}_{L^2 \to \dom^2_{N,k}} \lesssim \smallang{k}.
\end{equation}
Thus we obtain by interpolating \eqref{resest:1} and \eqref{resest:2}
$$
\norm{\chi R_{N}(k) \chi}_{L^2 \to \dom^1_{N,k}} \lesssim 1.
$$

Now once again if $(\Lap+k^2) u=-f$ then $(\Lap+k^2)
(-\Lap+k^2)^{\ell} u=-(-\Lap+k^2)^{\ell} f,$ hence by compact support
of $f$ the resolvent estimate yields
$$
\norm{\chi_0 (-\Lap+k^2)^\ell u}_{\dom^1_{N,k}} \lesssim \norm{(-\Lap+k^2)^{\ell} f},
$$
so that for $\chi$ with smaller support than $\chi_0$ we have
$$
\norm{\chi u}_{\dom^{2\ell+1}_{N,k}} \lesssim \norm{f}_{\dom^{2\ell}_{N,k}}.
$$
Interpolation now yields
$$
\norm{\chi R_{N}(k) \chi}_{\dom^{s}_{N,k}\to\dom^{s+1}_{N,k}} \lesssim 1
$$
for all $s\geq 0.$  Now duality (which exchanges $k$ and $-k$) yields the estimate for $s<0$ as
well.  This completes the proof of \eqref{neumann1}.

To prove \eqref{neumann2} we begin by using the Vainberg parametrix construction as
presented in \cite{TaZw:00} to establish the estimate for $s=0.$  In the
notation of that paper, we have (see the two displayed equations
preceding (3.5))
$$
R_N (k)\chi  =R^\sharp(k)(I+K(k))^{-1}
$$
where $K(k)$ is a holomorphic family of operators that is
shown to have small $L^2\to L^2$ operator norm for $k \gg 1,$ so that $(I+K(k))$ is
invertible there.  The parametrix $R^\sharp(k)$ is defined by
\begin{equation}\label{eq:vainbergparametrix}
R^\sharp(k)=\widetilde{R}(k)-\mathcal{F}_{t \to k}((1-\chi_c) V_a(t))
\end{equation}
where $\chi_c=1$ near $\Omega_-,$ and
$$
\widetilde{R}(k)=-i\mathcal{F}_{t \to k}(\zeta H(t)U(t) \chi).
$$
Here $\chi$ (also called $\chi_a$ in \cite{TaZw:00}) is a cutoff equal to $1$ in a neighborhood of
$\Omega_-,$ $H(t)$ is the Heaviside function, 
$$
U(t) = \frac{\sin t\sqrt{-\Lap}}{\sqrt{-\Lap}}
$$
(sine propagator for the Neumann Laplacian),
and 
$\zeta$ is a cutoff
with
  \begin{equation*}
    \zeta (t,z) =
    \begin{cases}
      1 & t \leq |z| + T_{0} \\
      0 & t \geq |z|+T_{0}'
    \end{cases}
  \end{equation*}
  for some $T_{0}' \geq T_{0}$.  The term $V_a(t)$ is obtained by
  solving the free wave equation (i.e.\ with the obstacle removed)
  with forcing given by the error term $-[\Box, \zeta] U(t) \chi$ and
  zero Cauchy data.  Happily, its analysis will be of no concern here,
  as the factor $(1-\chi_c)$ ensures that the corresponding term in
  \eqref{eq:vainbergparametrix} vanishes on $\Gamma.$

It thus suffices from \eqref{eq:vainbergparametrix} to know that $\gamma_+ \widetilde{R}(k) $ satisfies
the desired estimates.  To see this, note that if $f \in L^2,$ 
$\zeta U(t) f$ lies in $L^\infty([0, T];  H^1)$ for each $T<\infty,$ simply
by the functional calculus for the Neumann Laplacian and the
identification of $H^1(\Omega_+)$ with $\dom_{N}^1$. Now
Theorem \ref{theorem:tataru2}
implies that
$$
\gamma_+ \zeta H(t) U(t) \chi  f \in H^{\alpha}(\RR\times \Gamma).
$$
(Note that $\zeta$ has compact support in time in a neighborhood of
the obstacle, so there is no difference between local and global
results here; note also that the factor of $H(t)$ does not affect the
regularity since $U(0)=0$.)
We may now
Fourier transform this estimate by Lemma~\ref{lemma:FTSobolev} to get
$$
\gamma_+ \widetilde{R}(k) f \in H^\alpha_k
$$
when $f \in L^2.$

Finally, we extend to more general $s$ in the estimate
\eqref{neumann2}.  Fix Fermi normal coordinates near $\Gamma$ with $x$
denoting the normal variable (distance to $\Gamma$) and $y$ denoting
coordinates along $\Gamma.$  Let $V$ denote any smooth, compactly supported vector field
on $\Omega_+$ such that near $\Gamma,$ $V$ is of the form $\sum
a_j(x,y) \pa_{y_j}.$  Then $V$ can be restricted
to $\Gamma$ to give a (indeed, any arbitrary) vector field $V_\Gamma.$
Note that $[\Lap,V]$ is then a second order differential operator in
the $\pa_{y_j}$'s only near $\Gamma,$ hence we have (cf.\
\cite[p.407]{Ta:96a})
$$
[\Lap,V]: \dom^2_{N,k}\to L^2.
$$
Now if
$$
(\Lap+k^2)u=-f \in H^1_c(\Omega_+)
$$
with $u$ outgoing and $f$ compactly supported in some fixed set, then we compute
$$
V (\Lap+k^2)u=-V f,
$$
hence
$$
(\Lap+k^2) V u + [V,\Lap] u=-Vf.
$$
Thus, applying the Neumann resolvent and restricting gives
\begin{equation}\label{foobar}
V_\Gamma \gamma_+ u=\gamma_+ V u=-\gamma_+ R_N(k) Vf-\gamma_+ R_N(k) [V,\Lap] u.
\end{equation}
Now by the estimate \eqref{neumann2} for $s=0$ obtained above, we have
$$
\norm{\gamma_+ R_N(k) Vf}_{H^\alpha_k} \lesssim \norm{Vf}_{L^2} \lesssim \norm{f}_{H^1_k}.
$$
Moreover, \eqref{neumann1} yields $u \in \dom^2_{N,k}$ with norm
estimated by $\norm{f}_{H^1_k},$ hence
$$
\norm{[V,\Lap] u}_{L^2}\lesssim\norm{f}_{H^1_k}.  
$$
Thus, again by the $s=0$
estimate \eqref{neumann2}, 
$$
\norm{\gamma_+ R_N(k) [V,\Lap] u}_{H^\alpha_k} \lesssim \norm{f}_{H^1_k},
$$
and putting together our estimate for the two terms on the RHS of  \eqref{foobar}, we have obtained for any vector field $V_\Gamma$ on $\Gamma,$
\begin{equation}\label{restrictedsobolev1}
\norm{V_\Gamma \gamma_+ R_N(k) f}_{H^\alpha_k}  \lesssim \norm{f}_{H^1_k}.
\end{equation}
Also, just the fact that $f \in L^2$ and the $s=0$ estimate gives
\begin{equation}\label{restrictedsobolev2}
\smallang{k} \norm{ \gamma_+ R_N(k) f}_{H^\alpha_k} \lesssim \smallang{k}
\norm{f}_{L^2} \lesssim \norm{f}_{H^1_k}.
\end{equation}
Since $V_\Gamma$ was arbitrary, putting together
\eqref{restrictedsobolev1} and \eqref{restrictedsobolev2} yields, for
$f$ compactly supported in a fixed set,
$$
\norm{\gamma_+ R_N(k) f}_{H^{1+\alpha}_k}\lesssim \norm{f}_{H^1_k}.
$$
Interpolating with the $s=0$ estimate now yields \eqref{neumann2} for the whole range $s
\in [0,1].$
\end{proof}

\begin{theorem}\label{theorem:NtD}
Let $\Omega_+$ be nontrapping.  For each $\chi \in
C_c^\infty(\Oe),$ there exists $k_{0}$ so that solutions $u$ of the Helmholtz equation
\begin{equation}\label{eq:helmholtz}  \begin{aligned}
    (\lap + k^{2})u &= 0 \quad \text{ in }\Omega_+ \\
    \pa_n^+ u|_{\Gamma} &= g_N
\end{aligned}
\end{equation}
  satisfying the Sommerfeld radiation condition \eqref{eq:src} enjoy the bounds
  \begin{equation*}
    \norm{\chi u}_{H^\alpha_k(\Omega_+)}
  \lesssim \norm{g_N}_{L^2(\Gamma)}
  \end{equation*}
and
$$
\norm{\gamma_+ u}_{H^\beta_k(\Gamma)}\lesssim \norm{g_N}_{L^2(\Gamma)},
$$
for $k>k_0.$  Here $\alpha$ and $\beta$ are again given by
equation~\eqref{eq:alphabeta}. 
\end{theorem}

\begin{proof}
  Fix a cutoff function $\varphi(t)$ compactly supported in $(0,1)$ with $\int
  \varphi = 1$.  Suppose that $v_{\kappa}$ is the solution of
  \begin{align*}
    \Box v_{\kappa} &= 0, \\
   \pa_n v_{\kappa} |_{\Gamma} &= \varphi (t) \re^{-\ri\kappa t}g_N(y)=h_{\kappa}(t,y), \\
    v &= 0 \quad \text{ for } t < 0.
  \end{align*}
  
 Note that $\norm{h_\kappa}_{L^2(\RR\times \Omega_+)}\lesssim
 \norm{g_N}_{L^2(\Gamma)}$ for all $\kappa;$ this estimate and all those that
 follow have implicit constants that are, crucially, uniform in $\kappa.$

Let $I\subset \RR$ be an open interval containing $\supp \varphi.$
By Tataru's estimates in Theorem~\ref{theorem:tataru} (and the compact
support of $v_{\kappa}$ on $I \times \Omega_{+}$) we obtain
\begin{equation*}
    \Norm[H^{\alpha}(I\times \Omega_+)]{v_{\kappa}} \lesssim \Norm[L^2(I
    \times \Gamma)]{h_{\kappa}} \lesssim \norm{g_N}_{L^2( \Gamma)}.
  \end{equation*}
We further choose $\psi(t)$ a cutoff function supported in $I$ and equal to
$1$ on $\supp \varphi.$  Then we also have
 \begin{equation*}
    \Norm[H^{\alpha}(\RR\times \Omega_+)]{\psi v_{\kappa}} \lesssim \norm{g_N}_{L^2(\Gamma)}.
  \end{equation*}
Hence by Lemma~\ref{lemma:FTSobolev},
 \begin{equation*}
    \Norm[H_k^{\alpha}(\Omega_+)]{\mathcal{F}^{-1}(\psi v_{\kappa})} \lesssim \norm{g_N}_{L^2(\Gamma)}.
  \end{equation*}

Now since $v_\kappa$ satisfies the wave equation we have
$$
\Box (\psi v_\kappa) =[\Box,\psi] v_\kappa \in H^{\alpha-1}(\RR\times
  \Omega_+) \cap H^{\alpha-1}(\RR; L^2(\Omega_+))
$$
with the norm of the RHS again estimated by a multiple of $\norm{g_N}.$  (Note also that $\psi v_\kappa$ has
compact support in $\Omega_+.$)
Hence since\footnote{We are of course using the fact that $\alpha-1<0$ here.} $\ang{k}^{-\alpha+1}L^2(\Omega_+) \subset
\dom_{N,k}^{\alpha-1}(\Omega_+)$ we have
\begin{equation}\label{ekappa}
(\Lap+k^2) \mathcal{F}^{-1}(\psi v_\kappa) \equiv e_\kappa \in \dom_{N,k}^{\alpha-1}(\Omega_+),
\end{equation}
where
$$
\norm{e_\kappa}_{\dom^{\alpha-1}_{N,k}} \lesssim \norm{g_N}.
$$
Now the nontrapping estimates for the Neumann resolvent
as stated in Lemma~\ref{lemma:neumannresolvent}
tell us that if $R_N(k)$ denotes the outgoing Neumann
resolvent, then for $k \gg 1$ we have\footnote{We are using the
  identification of Neumann domains and Sobolev spaces for exponents in $[0,1].$}
\begin{align*}
\norm{\chi R_N(k)[e_\kappa]}_{H_k^\alpha}  &\lesssim
\norm{e_\kappa}_{H_k^{\alpha-1}}\\
&\lesssim\norm{g_N}.
\end{align*}
Now consider
\begin{equation}\label{eq:u}
u\equiv \mathcal{F}^{-1}(\psi v_\kappa)-R_N(k)[e_\kappa].
\end{equation}
By the foregoing discussion we have
$$
\norm{\chi u}_{H_k^\alpha} \lesssim \norm{g_N}.
$$
On the other hand, we have 
$$
(\Lap+k^2)u=0,
$$
by construction.  Moreover, since we used the Neumann resolvent in
constructing $u,$
\begin{align*}
\pa_n^+ u &=\pa_n^+\mathcal{F}^{-1}(\psi v_\kappa)\\&= \mathcal{F}^{-1}(\psi(t) \varphi(t)
\re^{\ri\kappa t}  g_N)\\&=\widehat\varphi(k-\kappa) g_N.
\end{align*}
Hence if we set $\kappa=k$ we obtain $u$ as the (unique) solution of
\eqref{eq:helmholtz} satisfying the radiation condition, and have obtained
the desired interior estimate.

To derive the boundary estimates, we use
Lemma~\ref{lemma:neumannresolvent} as well as
Theorem~\ref{theorem:tataru}.  The latter implies that 
$$
\gamma_+\mathcal{F}^{-1} (\psi v_k)\in H^\beta_k,
$$
hence by \eqref{eq:u} it suffices to consider the term $R_N(k)
[e_\kappa].$ Returning to the definition \eqref{ekappa} of $e_\kappa$
we note that we can in fact write
$$
e_\kappa=\mathcal{F}^{-1}(\pa_t f^1_\kappa+f^2_\kappa),\text{ where }
f^i_{\kappa} \in H^\alpha_c(I \times \Omega_+).
$$
Thus we obtain a slightly refined estimate on $e_\kappa$:
$$
e_\kappa \in \smallang{k} H^\alpha_k(\Omega_+).
$$
Now since $\alpha\in (0,1),$ the estimate \eqref{neumann2}
of Lemma~\ref{lemma:neumannresolvent} yields an estimate on
$$\gamma_+ R_N(k) [e_\kappa]\in \smallang{k} H_{k}^{2\alpha}(\Omega_+)
\subset H_{k}^{2\alpha-1}(\Omega_+),$$
as desired.  (Recall that $2\alpha-1=\beta.$)
\end{proof}

\begin{corollary}\label{corollary:ntdshifted1}
With notation as above,
$$\norm{\chi u}_{H^1_k(\Omega_+)} \lesssim \abs{k}^{1-\alpha}
\norm{g_N}_{L^2},\quad \abs{k}\geq k_0.$$
\end{corollary}
\begin{proof}
This follows from combining the bounds in Theorem \ref{theorem:NtD} with the result of Lemma \ref{lem:H1L2}.
\end{proof}

\begin{corollary}\label{cor:NtD2}
With notation as above, we have
$$
\norm{\gamma_+ u}_{H^1_k} \lesssim \abs{k}^{1-\beta}
\norm{g_N}_{L^2},\quad \abs{k}>k_0.
$$
\end{corollary}
\begin{proof}
By the second part of Lemma~\ref{lem:Necas}, we have
$$
\norm{\gamma_+ u}_{H^1_k} \lesssim \norm{g_N} + k \norm{\gamma_+ u} + \norm{\chi_R u}_{H^1_k},
$$
hence the results follows from the estimates on the second and third
terms given above in Theorem~\ref{theorem:NtD} and
Corollary~\ref{corollary:ntdshifted1} respectively. 
\end{proof}

\bpf[Proof of Theorem \ref{thm:2}]
The bound \eqref{eq:NtD2} follows from combining the bounds in
Corollaries \ref{corollary:ntdshifted1} and \ref{cor:NtD2} with Lemma
\ref{lem:Necas} (note that $1-\beta>1-\alpha$ in both the general and
positive curvature cases). The bound \eqref{eq:NtD1} then follows by Lemma \ref{lem:inter}.
\epf

\section{The interior impedance problem}\label{sec:IIP}

\subsection{Motivation}
\label{sec:iip-mot}

For readers unfamiliar with the numerical analysis literature on the Helmholtz equation, we explain in this section why the interior impedance problem is of interest to numerical analysts (independent from the fundamental role it plays in the theory of integral equations for exterior problems, which we discuss in \S\ref{sec:1-3} and \S\ref{sec:int}). 

The majority of research effort concerning numerical methods for Helmholtz problems is focused on solving scattering/exterior problems in 2- or 3-d (such as the exterior Dirichlet and Neumann problems considered in \S\ref{sec:DtN} and \S\ref{sec:NtD}).
Boundary integral equations (BIEs) are in many ways ideal for this task, since they reduce a $d$-dimensional problem on an unbounded domain to a $(d-1)$-dimensional problem on a bounded domain. However there is still a very large interest in domain-based (as opposed to boundary-based) methods such as the finite element method, partly because these are usually much easier to implement than BIEs and partly because these domain-based methods usually generalize to the case when $k$ is variable (as occurs, for example, in seismic-imaging applications). 

When solving scattering problems with domain-based methods, one must come to grips with unbounded nature of the domain. This is normally done by
truncating the domain: one chooses a (large) bounded domain
$\tOmega \supset \Oi$, imposes a boundary condition on $\partial
\tOmega$, and then solves the BVP in
$\tOmega\setminus\overline{\Oi}$. If $\tOmega$ is a ball, one can
choose the boundary condition on $\partial \tOmega$ such that the
solution to the BVP in $\tOmega\setminus\overline{\Oi}$ is precisely
the restriction of the solution to the scattering problem---one does
this by using the explicit expression for the solution of the
Helmholtz equation in the exterior of a ball, and the relevant
boundary condition on $\partial \tOmega$ involves the so-called
Dirichlet-to-Neumann operator (see, e.g., \cite[\S3.2]{Ih:98} for more
details). Alternatively one can impose approximate boundary conditions
(often called \emph{absorbing boundary conditions} or
\emph{non-reflecting boundary conditions} since their goal is to
absorb any waves hitting $\partial \tOmega$ instead of reflecting them
back into $\tOmega$), the simplest such one being $\dudnw -\ri k u =0$
on $\partial \tOmega.$  This can be viewed this as an approximation to the
radiation condition \eqref{eq:src}.

Therefore, in the simplest case, truncating a Helmholtz BVP in an unbounded domain yields a BVP for the Helmholtz equation in the  annulus-like region $\tOmega\setminus \overline{\Oi}$, with an impedance boundary condition on  $\partial \tOmega$, and either a Dirichlet or Neumann boundary condition on $\Gamma$. Without a $k$-explicit bound on the solution of this BVP, a fully $k$-explicit analysis of any numerical method is impossible, and therefore the problem of finding $k$-explicit bounds on the solution of this truncated problem was considered in \cite{He:07}, \cite{ShWa:05}.

Going one step further, although the geometry of the scatterer plays an important role in determining the behaviour of the solution, many features of numerical methods for the Helmholtz equation (such as whether the so-called pollution effect occurs) can be investigated without the presence of a scatterer at all; this then leads to considering the Helmholtz equation posed in a bounded domain with an impedance boundary condition, i.e., the IIP (and the impedance boundary condition can then be viewed as a way of ensuring that the solution of the BVP is unique for all $k$). The problem of finding $k$-explicit bounds on the solution of the IIP was therefore considered in \cite{FeSh:94}, \cite{Me:95}, \cite{CuFe:06}, \cite{EsMe:12}, and \cite{Sp:14}.

Midway between, in some sense, the truncated scattering problem and the IIP are BVPs posed on bounded domains, where impedance boundary conditions (or more sophisticated absorbing boundary conditions) are posed on part of the boundary, and Dirichlet or Neumann boundary conditions are posed on the rest. The most commonly-studied such problem is the Helmholtz equation in a rectangle with impedance boundary conditions on one side and Dirichlet boundary conditions on the other three, motivated by the physical problem of scattering by a half plane with a rectangular indent (or ``cavity"). Bounds on this problem were obtained in \cite{BaYuZh:12} and \cite{LiMaSu:13}, and the recent paper \cite{DuLiSu:15} seeks to determine the optimal dependence on $k$ via numerical experiments.

\subsection{Interior impedance estimates}
\label{sec:iip-est}

We begin with a result about uniqueness of solutions of the IIP for
complex values of the spectral parameter $k.$

\ble[Uniqueness of the IIP]\label{lem:IIP1}
Consider the IIP \eqref{eq:IIP} with 
\beq\label{eq:5eta}
\eta(\bx) = \notalpha(x) k + \ri \notbeta(x),
\eeq
where $\notalpha,\notbeta$ are real-valued $C^\infty$ functions on $\Gamma$.

\bit
\item[(i)] If there exists an $\notalpha_->0$ such that 
\beq\label{eq:51}
\notalpha(x)\geq \notalpha_- >0 \quad\tfa x\in \Gamma,
\eeq
and $\notbeta(x)\geq0$ on $\Gamma$, then the solution of the IIP is unique for all $k\neq 0$ with $\Im k \geq 0$. 
\item[(ii)] If there exists an $\notalpha_->0$ such that \eqref{eq:51} holds and there also exists a $\notbeta_->0$ such that 
\beq\label{eq:51a}
\notbeta(x)\geq \notbeta_- >0 \quad\tfa x\in \Gamma,
\eeq
then the solution of the IIP is unique for all $k$ with $\Im k\geq 0$ (i.e. we now also have uniqueness when $k=0$).
\eit
\ele

\bpf
If $u$ is the solution of the homogeneous IIP (i.e. $f=0$ and $g=0$) then applying Green's identity and using the impedance boundary condition we find that
\beqs
\ri k \int_\Gamma \notalpha |\gamma u|^2 - \int_\Gamma \notbeta |\gamma u|^2 - \int_\Omega \ngus + k^2 \int_\Omega \nus=0.
\eeqs
Therefore, taking real and imaginary parts, and writing $k= k_R+ \ri k_I$ with $k_R, k_I\in \Rea$, we have
\beq\label{eq:52}
-k_I \int_\Gamma \notalpha  |\gamma u|^2 - \int_\Gamma \notbeta |\gamma u|^2 - \int_\Omega \ngus + (k_R^2 - k_I^2) \int_\Omega \nus=0,
\eeq
and 
\beq\label{eq:53}
k_R \int_\Gamma \notalpha  |\gamma u|^2 + 2k_R k_I \int_\Omega \nus=0
\eeq
Proof of (i): if $k_R\neq 0$ and $k_I\geq 0$, then 
using the assumption \eqref{eq:51} on $\notalpha$ in \eqref{eq:53} we see that $\gamma u=0$. The impedance boundary condition then implies that $\dnu=0$, and thus Green's integral representation (see, e.g., \cite[Theorem 7.5]{Mc:00}) implies that $u=0$ in $\Omega$. If $k_R=0$ and $k_I>0$, then using both the assumption \eqref{eq:51} on $\notalpha$ and the assumption that $\notbeta$ is non-negative in \eqref{eq:52}, we see that $u=0$ in $\Omega$.

Proof of (ii): from Part (i) we only need to consider the case when $k=0$. Using the assumption \eqref{eq:51a} in \eqref{eq:52}, we see that $\gamma u=0$ on $\Gamma$, and then $u=0$ in $\Omega$ follows from the steps above.
\epf

We now prove Theorem \ref{thm:3} by employing the 
estimates of Bardos--Lebeau--Rauch \cite{BaLeRa:92} for the wave
equation with the damping boundary condition, i.e. 
\begin{subequations}\label{eq:BLR1}
\begin{align}
\Box v &=0 \text{ on } \Omega,\\
(\pa_n+\notalpha\gamma\pa_t+\notbeta\gamma) v &=0 \text{ on } \Gamma\label{eq:BLR1b}
\end{align}
\end{subequations}
where $\notalpha, \notbeta$ are smooth, real-valued functions on $\Gamma$ with $\notalpha$ strictly positive and $\notbeta$ nonnegative.

First we give a short proof of the standard energy estimate for the wave equation, but now considering the boundary condition \eqref{eq:BLR1} instead of the usual Dirichlet or Neumann ones.

\begin{lemma}\label{lemma:impedanceenergy}
Let $F\in L^2(\RR\times \Omega)$ and $G\in L^2(\RR\times \Gamma)$ be supported in $t>0$ and let $v$ solve
\begin{align*}
\Box v&=F\text{ on } \Omega,\\
(\pa_n+\notalpha\gamma\pa_t+\notbeta\gamma) v&=G \text{ on } \Gamma,\\
v&=0 \text{ for } t\leq 0,
\end{align*}
where $\notalpha, \notbeta$ are smooth, real-valued functions on $\Gamma$ with $\notalpha$ strictly positive and $\notbeta$ nonnegative.
Then for any $T$
$$
\norm{v_t}^2+ \norm{\nabla v}^2 + \norm{\notbeta^{1/2}\gamma v}^{2} \rvert_{t=T} \leq C_T
\big(\norm{F}^2_{L^2([0,T]\times \Omega)}+\norm{G}^2_{L^2([0,T]\times \Gamma)}\big).
$$
\end{lemma}
\begin{proof}
Without loss of generality we can assume that $F$ and $G$ are both real. Multiplying $\Box v=F$ with $v_t$ and integrating over $\Omega$ we find 
\begin{align}\label{eq:5-1}
\pdiff{}{t}\left( 
\int_\Omega \big(|\gv|^2 + (v_t)^2\big) + \int_\Gamma \notbeta (\gamma v)^2
\right)
= -\int_\Gamma \notalpha (\gamma v_t)^2 + \int_\Gamma G\, \gamma v_t + \int_\Omega F \, v_t.
\end{align}
Using the Cauchy-Schwarz inequality on the second term on the RHS of
\eqref{eq:5-1} and recalling that $\notalpha$ is strictly positive, we
see that we can bound the first two terms by a multiple of $\int_\Gamma G^2$. The other term on the RHS of \eqref{eq:5-1} is bounded by $\half(\int_\Omega F^2 + \int_\Omega (v_t)^2)$, and the result then follows from Gronwall's inequality (see, e.g., \cite[\S7.2.3]{Ev:98}), using the fact that $\notbeta\geq0$.
\end{proof}

In the proof of Theorem~\ref{thm:3} below, the crucial microlocal
ingredient will be the estimates on the wave equation with impedance
boundary condition obtained by Bardos--Lebeau--Rauch~\cite{BaLeRa:92}.
These estimates involve a key geometric hypothesis, which is that
every generalized bicharacteristic in the sense of
Melrose--Sj{\"o}strand~\cite{MeSj:82} eventually hits the boundary
(or, in the more general setting of \cite{BaLeRa:92}, the control
region) at
a point that is \emph{nondiffractive} as defined in
\cite[p.1037]{BaLeRa:92}.\footnote{Note that the negation of
  ``nondiffractive'' in this sense is not the same as ``diffractive''
  in the sense of \cite{MeSj:82}.} In our simple case of compact Euclidean
domains, we remark that these
hypotheses are always satisfied:

\begin{lemma}
  \label{lemma:nondiffractive}
  If $\Omega_{-}\subset \reals^{n}$ is a compact domain with smooth
  boundary, then every generalized bicharacteristic eventually hits
  the boundary at a nondiffractive point.
\end{lemma}

\begin{proof}
  We first observe that a generalized bicharacteristic in a compact
  Euclidean domain must eventually change momentum.  Adopting the
  notation of H{\"o}rmander~\cite[Definition 24.3.7]{Ho:83}, we claim
  that the only way the momentum can change along a generalized
  bicharacteristic is when it hits the boundary at a point in $\hyp
  \cup \gla \setminus \gla_{d}$.  Here $\hyp$ denotes the ``hyperbolic
  points'' at which there is transverse reflection from the boundary,
  while $\gla \setminus \gla_{d}$ denotes the set of glancing points
  that are not diffractive.  To prove this assertion, we note that in
  the interior and at diffractive points (which together constitute
  the remaining parts of the characteristic set), we have $\gamma'(t)
  = H_{p}(\gamma(t))$, where $\gamma$ denotes the bicharacteristic and
  $H_{p}$ the Hamilton vector field, which in this case is the
  constant vector field $\xi \cdot \pd[x]$ in $T^{*}\reals^{n}$ (cf.\
  Chapter 24 of \cite{Ho:83}).

  Now we further note that on $\gla \setminus \gla_{d}$, we have
  $\gamma '(t) = H_{p}^{G}(\gamma (t))$ with $H_{p}^{G}$ the ``gliding
  vector field'' of Definition~{24.3.6} in \cite{Ho:83}.  This vector
  field \emph{still} agrees with $H_{p}$ unless $\gamma(t) \in \gla^{2}$, the
  points where contact with the boundary is exactly second-order.  On
  the other hand, the ``gliding points,'' $\gla_{g} \equiv
  \gla^{2}\setminus \gla_{d}$, are nondiffractive by the definition of
  Bardos--Lebeau--Rauch, since the second derivative of the boundary
  defining function is strictly negative along the flow at such points
  (cf.\ Definition~{24.3.2} of \cite{Ho:83}).
\end{proof}

\begin{proof}[Proof of Theorem \ref{thm:3}]
We begin by dealing with the case when $\notalpha$ is positive.
 By \cite{BaLeRa:92}, if $v$ satisfies \eqref{eq:BLR1} with initial data
  in the energy space, then all energy norms of $v$ enjoy exponential decay
  as $t \to \infty$.
Indeed, \cite[Theorem~5.5 and Proposition~5.3]{BaLeRa:92} prove this result
for the case when $\notbeta$ is nonnegative, and then the result for
$\notbeta\equiv 0$ follows from \cite[Theorem~5.6]{BaLeRa:92}, but we
emphasize that in this latter case it is \emph{just} the energy norm
$$
\norm{v_t}^2+\norm{\nabla v}^2+\smallnorm{\notbeta^{1/2} \gamma v}^2
$$
that converges to zero, while the value of the solution may converge to a
nonzero constant, since this norm does not in general control the
$L^2$ norm.


We let $v_\kappa$ denote the (unique) solution to the wave equation on
$\RR\times \Omega_-$ satisfying 
\begin{subequations}\label{eq:wave_shift}
\begin{align}
\Box v_\kappa &= \re^{-\ri\kappa t} \varphi(t) f,\\
(\pa_n+\notalpha\gamma\pa_t+\notbeta\gamma)v_\kappa &= \re^{-\ri\kappa t} \varphi(t) g,\\
v_\kappa(t,x)&=0,\quad t<0.
\end{align}
\end{subequations}
where $\varphi$ is a cutoff compactly supported in $(0,1)$ with $\int \varphi=1.$
Then the standard energy estimate
proved in
Lemma~\ref{lemma:impedanceenergy} yields 
$$
\norm{(v_\kappa)_t}^2+\norm{\nabla v_\kappa}^2+ \smallnorm{\notbeta^{1/2}
  \gamma v_{\kappa}}^2\big \rvert_{t=1}\lesssim \N{f}^2_{L^2(\Omega)}+\N{g}^2_{\LtG}.
$$
Now since $v_{\kappa}$ satisfies the homogeneous wave equation for $t
\geq 1$ with initial data at $t=1$ controlled as above,
\cite[Theorem~5.5]{BaLeRa:92}  yields, for some $\delta>0,$
\beq\label{eq:BLR2}
\norm{(v_\kappa)_t}^2+\norm{\nabla v_\kappa}^2 +
\smallnorm{\notbeta^{1/2} \gamma v_{\kappa}}^2
\leq C \re^{-\delta t}
\big(\N{f}^2_{L^2(\Omega)}+\N{g}^2_{\LtG}\big),\quad t>0.
\eeq
Fourier transforming \eqref{eq:wave_shift} gives 
\begin{align*}
(\Lap+k^2)\mathcal{F}^{-1} v_\kappa&=-\widehat \varphi(k-\kappa) f,\\
(\pa_n- \ri k\notalpha\gamma+\notbeta\gamma)
\mathcal{F}^{-1} v_\kappa &= \widehat \varphi(k-\kappa) g.
\end{align*}
Since
\beqs
\N{\cF^{-1} v}_{L^2_x} \leq \N{v}_{L^1_t\, L^2_x}
\eeqs
the exponential decay estimate \eqref{eq:BLR2} implies that
$$
\norm{\nabla \mathcal{F}^{-1} v_\kappa}+\abs{k} \norm{\nabla \mathcal{F}^{-1} v_\kappa}\lesssim \norm{f}+\norm{g}, \quad k\in \RR;
$$
here we have made no use of the boundary term on the LHS of \eqref{eq:BLR2}.
If the stronger Assumption~\ref{ass:eta2} holds, we employ the more precise
version of our Fourier transformed estimates:
$$
\norm{\nabla \mathcal{F}^{-1} v_\kappa}^2+\abs{k}^2 \norm{ \mathcal{F}^{-1}
  v_\kappa}^2+ \smallnorm{\notbeta^{1/2} \gamma \mathcal{F}^{-1} v_\kappa}^2
\lesssim \norm{f}^2+\norm{g}^2,\ k \in \RR.
$$
By the Poincar{\'e}-Wirtinger inequality\footnote{Note that in
  employing the Poincar\'e-Wirtinger inequality, we may estimate the
  average value of $u$ by a multiple of $\norm{\nabla u} +
  \norm{\gamma u}$ by writing it as a multiple of $\int u\, \nabla
  \cdot x \, dx$ and integrating by parts.} and the positivity of
$b,$ the left side controls
$$
\norm{\nabla \mathcal{F}^{-1} v_\kappa}^2+\ang{k}^2 \norm{ \mathcal{F}^{-1}
  v_\kappa}^2
$$
 even at $k=0$, giving us
the stronger estimate (cf.\ discussion on
pps.1060--1061 of \cite{BaLeRa:92}):
$$
\norm{\nabla \mathcal{F}^{-1} v_\kappa}+ \ang{k} \norm{\mathcal{F}^{-1} v_\kappa}\lesssim \norm{f}+\norm{g},\ k \in \RR.
$$

Taking $\kappa=k$ makes $u\equiv v_k$ the solution of the IIP
\eqref{eq:IIP} and yields the asserted estimate \eqref{eq:thm3} when
$\notalpha$ is strictly
positive.  This concludes the proof for $\notalpha$ strictly positive.

When $\notalpha$ is strictly negative, the sign convention of the Fourier
transform and the signs of the exponents in \eqref{eq:wave_shift} can
both be changed to give the correspond estimate. (Alternatively, by
taking the complex conjugate of the BVP \eqref{eq:IIP}, we can prove
\eqref{eq:thm3} when the boundary condition \beqs (\pa_n + \ri k\notalpha
\gamma + \notbeta\gamma )u=g \eeqs is imposed. If $\notalpha$ is strictly
negative, then we apply the bound above with $\notalpha$ replaced by
$-\notalpha$, and this yields the desired result.)
\end{proof}


We now prove Corollary \ref{cor:ItD}, regarding the
impedance-to-Dirichlet map, by using Theorem \ref{thm:3} in
conjunction with a simple energy estimate.

\bpf[Proof of Corollary \ref{cor:ItD}]
Reiterating the integration by parts used to obtain
Lemma~\ref{lem:IIP1} but now including the inhomogeneities, we find
that applying Cauchy-Schwarz to our expression for the imaginary part
of $\int_\Omega f\,\overline{u}$ 
yields for $k \in \RR$
$$
k \norm{\sqrt{a} \gamma u}^2_{L^2(\Gamma)} \lesssim
\abs{\int_\Gamma g\, \overline{\gamma u}} 
+ \abs{\int_\Omega f\, \overline{u}}.
$$
Now applying Cauchy-Schwarz and the estimates of Theorem~\ref{thm:3}
to the resulting $\norm{u}$
term on the RHS gives the desired estimate on $k^2 \norm{\gamma u}^2.$

The corresponding estimate for $\nT (\gamma u)$ follows from the analogous
estimate to Lemma~\ref{lem:Necas}(ii) for bounded domains (the same proof
employed by Ne\v{c}as applies).

If $b$ is strictly positive, we obtain the stronger estimate at $k=0$ by
examining the real rather than the imaginary part of
$\int_\Omega f\,\overline{u}$
to estimate $\int_\Gamma b\smallabs{\gamma u}^2.$
\epf

\bpf[Proof of Corollary \ref{cor:infsup}]
We follow the argument in, e.g., \cite[Theorem 2.5]{EsMe:12}, \cite[text between (3.3) and (3.4)]{ChMo:08}. The variational formulation of the IIP is 
\beq\label{eq:new1}
\text{find  } u \in H^1(\Omega) \text{  such that  } a(u,v) =F(v) \,\, \tfa v\in H^1(\Omega),
\eeq
where 
\beq\label{eq:sesqui}
a(u,v) := \int_\Omega \gu \cdot \overline{\gv} - k^2 u\,\overline{v} - \ri k \int_\Gamma \notalpha \,\gamma u \,\overline{\gamma v} + \int_\Gamma \notbeta \,\gamma u\,\overline{\gamma v},
\eeq
and 
\beq\label{eq:F}
F(v):= \langle f, v\rangle_\Omega + \langle g, \gamma v\rangle_\Gamma, 
\eeq
where $\langle\cdot,\cdot\rangle_\Omega$ and $\langle\cdot,\cdot\rangle_\Gamma$ denote the duality pairings on $\Omega$ and $\Gamma$, respectively.
Define the sesquillinear form $a_0(\cdot,\cdot)$ by
\beq\label{eq:sesqui2}
a_0(u,v) := \int_\Omega \gu \cdot \overline{\gv} + k^2 u\,\overline{v} - \ri k \int_\Gamma \notalpha \,\gamma u \,\overline{\gamma v} + \int_\Gamma \notbeta \,\gamma u\,\overline{\gamma v}.
\eeq
Furthermore, define $u_0\in H^1(\Omega)$ as the solution of the variational problem $a_0(u_0,v)= F(v)$ for all $v\in H^1(\Omega)$, and define $w\in H^1(\Omega)$ as the solution of the variational problem $a(w,v)= 2k^2 \int_\Omega u_0\, \overline{v}$ for all $v\in H^1(\Omega)$. These definitions imply that the solution of \eqref{eq:new1} satisfies $u=u_0+w$.

Since $\notbeta$ is nonnegative, $\Re a_0(v,v) = \N{v}^2_{H^1_k(\Omega)}$; thus, by the Lax--Milgram lemma, $\N{u_0}_{H^1_k(\Omega)}\lesssim \N{F}_{(H^1_k(\Omega))'}$. The definition of $w$ implies that $w$ satisfies the IIP with $g=0$ and $f= 2k^2 u_0$, and thus the bound \eqref{eq:thm3} implies that $\N{w}_{H^1_k(\Omega)}\lesssim k^2 \N{u_0}_{L^2(\Omega)}$. Combining these bounds on $u_0$ and $w$, we obtain
\beq\label{eq:new2}
\N{u}_{H_k^1(\Omega)} \lesssim |k|\N{F}_{(H^1_k(\Omega))'}.
\eeq
The result on the inf-sup constant \eqref{eq:infsup2} then follows from, e.g., \cite[Theorem 2.1.44]{SaSc:11}. The bound \eqref{eq:infsup} follows from \eqref{eq:new2} using the definition of $F$ \eqref{eq:F}.
\epf

\bpf[Proof of Corollary \ref{cor:H2}]
The bound \eqref{eq:H2} follows from combining the bounds \eqref{eq:thm3} and 
\beqs
\N{u}_{H^2(\Omega)} \lesssim \N{\Delta u}_{L^2(\Omega)} + \N{u}_{H^1(\Omega)} + \N{\dnu}_{H^{1/2}(\Gamma)},
\eeqs
where the latter is proved in, e.g., \cite[Theorem 2.3.3.2, page 106]{Gr:85}.
\epf

We now impose the \emph{homogeneous} impedance boundary condition, and
consider the operator $\Rimp(k):L^2(\Omega) \to L^2(\Omega)$ defined
by $\Rimp(k) f=u$ where
$u$ is the solution to $(\Lap+k^2)u=f$ satisfying
$(\partial_n-\ri \eta\gamma) u=0.$

Following the discussion in \S\ref{sec:introduction}, we now proceed with
the assumptions that $\notalpha$ and $\notbeta$ are both strictly positive
(i.e. \eqref{eq:51} and \eqref{eq:51a} hold), so that $\Rimp(k)$ is well defined for all $\Im k\geq 0$.

We break the proof of Theorem \ref{thm:pole} down into several steps; the first step is to prove that $\Rimp(k)$ is holomorphic on $\Im k> 0$.

\ble[Analyticity for $\Im k >0$]\label{lem:IIP2}
The operator family $\Rimp(k):\LtO\rightarrow \LtO$ with boundary condition
\beq\label{eq:54}
\dnu - \ri (k\notalpha + \ri \notbeta) \gamma u=0,
\eeq
where $\notalpha,\notbeta$ are real-valued $C^\infty$ functions with
$\notalpha$  strictly positive on $\Gamma$ and $\notbeta$ nonnegative, is holomorphic on $\Im
k>0$.
\ele

\bpf
First note that the standard variational formulation of the IIP satisfies a G\aa rding inequality. Indeed, the sesquilinear form is given by \eqref{eq:sesqui}
and so, since $\notbeta$ is non-negative and $\Im k>0$, we have
\beqs
\Re a(v,v) + (1 +k^2) \N{v}^2_{\LtO} \geq \N{v}^2_{H^1(\Omega)}
\eeqs
(note that we are using the unweighted norm on $H^1(\Omega)$ since we are allowing for $k$ to be equal to zero).
Fredholm theory then gives us well-posedness of the BVP as a consequence of the uniqueness result in Lemma \ref{lem:IIP1} (see, e.g., \cite[Theorems 2.33, 2.34]{Mc:00}).
Analyticity follows by applying the Cauchy-Riemann
operator $\pa/\pa \overline{k}$ to the equations
$(\Lap+k^2)u=f$ and $\dnu - \ri (k\notalpha + \ri \notbeta) \gamma
u=0:$ we find that $\pa u /\pa \overline{k}$ must satisfy
the IIP with zero interior and boundary data, hence by the uniqueness
proved above, it must vanish.
\epf

We now use a simple perturbation argument to get the existence of a
pole-free strip beneath the real axis.


\bpf[Proof of Theorem \ref{thm:pole}] Lemma \ref{lem:IIP2} states that
$\Rimp(k)$ is holomorphic on $\Im k> 0$, while Theorem~\ref{thm:3}
yields the estimate \eqref{eq:55} for all $k \in \RR$ (crucially using
Assumption~\ref{ass:eta2}).  We can now perturb using this estimate to
extend to analyticity below the real axis, but we will need to
consider the full inverse map on both interior and boundary data (and
in so doing, will in fact prove a stronger result than stated,
involving both interior and boundary inhomogeneities).  For the
(unique) solution of the IIP
$$
(\Lap+k^2) u=f,\quad (\pa_n -\ri k\notalpha
\gamma + \notbeta\gamma )u=g
$$
we set 
$$
\begin{pmatrix}
u \\ \gamma u
\end{pmatrix}=
\tRimp(k) \begin{pmatrix}
f\\ g
\end{pmatrix}.
$$
Then Corollary~\ref{cor:ItD} shows that for $k \in \RR,$
$$
\tRimp(k): L^2(\Omega) \oplus L^2(\Gamma) \to H^1_k(\Omega) \oplus H^1_k(\Gamma).
$$

Now for $z \in \CC$ we may try to solve
$$
(\Lap+(k+z)^2) u=f,\quad (\pa_n - \ri (k+z)\notalpha
\gamma + \notbeta\gamma )u=g
$$
by perturbation; we easily see that this is equivalent to
$$
(\Lap+k^2) u = f-(2kz+z^2) u,\quad (\pa_n - \ri k\notalpha
\gamma + \notbeta\gamma )u=g+\ri z \notalpha \gamma u.
$$
Hence, applying $\tRimp(k),$ we wish to solve
$$
\begin{pmatrix}
u \\ \gamma u 
\end{pmatrix}=
\tRimp(k) \begin{pmatrix}
f-(2kz+z^2)  u\\
g+\ri z \notalpha \gamma u
\end{pmatrix} =
\tRimp(k) \begin{pmatrix} f \\ g \end{pmatrix} - \tRimp(k) M(z) \begin{pmatrix}  u \\
  \gamma u
\end{pmatrix},
$$
where
$$
M(z) =\begin{pmatrix}
2kz +z^2 &0 \\ 0 & -\ri z\notalpha
\end{pmatrix}
$$
We can solve this by Neumann series (hence for a holomorphic solution
with the same $k$-dependent estimates as on the real axis) so
long as, say,
$$
\norm{\tRimp(k) M(z)}_{L^2\oplus L^2 \to L^2\oplus L^2}<1/2.
$$
Since $\tRimp(k)$ has norm bounded by $C \ang{k}^{-1}$ on $L^2 \oplus
L^2,$ this requires only that $\smallabs{z} \leq \epsilon$ for some
$\epsilon>0.$  Restricting to the case $g=0$ gives the stated result.
\epf

\ble[Sharpness of \eqref{eq:thm3} when $f=0$ and $\Omega$ is a ball]\label{lem:ball}
In $\RR^d$ for any $d\geq 2$ there exist families of solutions $u$ to
the interior impedance problem in the unit ball $B^d$ with boundary
inhomogeneity $g:$
\beq
\Delta u + k^2 u = 0 \quad\mbox{ in } B^d \quad\tand\quad
\dnu - \ri \eta \gamma u = g \quad \mbox{ on } S^{d-1}
\eeq
with 
$$
k\norm{u}_{L^2(B^d)}\gtrsim \norm{g}_{L^2(S^{d-1})}.
$$
\ele
\begin{proof}
Fix any spherical harmonic $\varphi(\theta)$ on $S^{d-1}_\theta$ with
eigenvalue $-\mu^2.$  Then the function
$$
u(r,\theta)\equiv r^{1-d/2} J_\nu(kr)\varphi(\theta)
$$
solves the Helmholtz equation in $B^d$ if we set
$$
\nu=\frac 12 \sqrt{(d-2)^2+ 4\mu^2}.
$$
We will let $k\to \infty$ while letting $\mu$ (and hence $\nu$) remain
fixed.

The function $u$ thus satisfies the
IIP (with $\eta=k$) where
$$
g\equiv (\pa_r-\ri k)u\lvert_{r=1}.
$$

Now we let $k \to \infty$ and examine the asymptotics of $u$ and $g.$
Since (see, e.g., \cite[Equation 10.17.3]{Di:13} for the standard Bessel function asymptotics
employed here)
$$
u = \varphi (\theta) r^{1-d/2} \sqrt{\frac 2{\pi kr}} \big(\cos \omega+O((rk)^{-1})\big)
$$
with
$$
\omega \equiv rk-\frac 12 \nu \pi-\frac 14 \pi
$$
we have
\beq\label{interiorlb}
\norm{u}_{L^2}  \gtrsim k^{-1/2}
\eeq
as $k \to\infty$ with $\nu$ fixed.  On the other hand, using the
asymptotic expansion of $J_\nu'$ yields
$$
\pa_r u =-\varphi(\theta) r^{1-d/2} k \sqrt{\frac{2}{\pi k r}} \big(\sin
  \omega+O(k^{-1})\big),
$$
hence at $r=1$ we have
$$
(\pa_r-\ri k)u\sim \varphi(\theta)
\sqrt{\frac{2k}{\pi}}\big(\cos\omega_0+\ri \sin \omega_0\big)
$$
with $\omega_0= k-\frac 12 \nu \pi-\frac 14 \pi.$  Thus, 
$$
\norm{(\pa_r-\ri k)u}_{L^2(S^{d-1})}\sim C k^{1/2}.
$$
Comparing to \eqref{interiorlb} yields the desired estimate.
\end{proof}

\begin{remark}[Extension to inhomogeneous problems]
The results of this section hold equally well, with identical proofs,
if we generalize the flat Laplacian to an inhomogeneous and/or
anisotropic operator with smooth coefficients such as
$$
\sum \partial_i a^{ij}(x) \partial_j,
$$
with $a^{ij}(x)$ strictly positive definite.
The only difference is that we then need to impose an auxiliary
geometric hypothesis, as Lemma~\ref{lemma:nondiffractive} no longer
applies. In this setting, motion along straight lines is replaced by
the Hamiltonian dynamical system
\begin{equation}\label{hamdyn}
\begin{aligned}
\dot x_i(t) &=\sum a^{ij}(x) \xi_j\\
\dot \xi_i(t) &=-\frac 12 \sum \frac{ \pa a^{kl}(x)}{\pa x_i} \xi_k\xi_l.
\end{aligned}
\end{equation}
It may easily be the case that trajectories of this system---which are
lifts to the cotangent bundle of geodesics with respect to the
Riemannian metric $a^{ij}(x)$---
fail to
reach the boundary at a nondiffractive point or indeed at all (e.g.\
$a^{ij}$ may be locally isometric in some
region to more than half of a round sphere).  Thus, we simply need to impose
geometric control by the boundary as a hypothesis: we insist that all
trajectories of \eqref{hamdyn} do reach the boundary at a
nondiffractive point.  The rest of our results then follow as in the
flat case.
\end{remark}


\section{Boundary integral equations for the exterior Dirichlet and Neumann problems}\label{sec:int}

In this section we derive both the integral equation \eqref{eq:CFIE} for the solution of the exterior Dirichlet problem and the analogous equation for the solution of the exterior Neumann problem. 
We then give a new proof of the decomposition \eqref{eq:key} (which is more intuitive than the proof in \cite{ChGrLaSp:12}), and we then prove an analogous decomposition for the integral equation for the Neumann problem. 

We note that there are now many good texts discussing the theory of integral equations for the Helmholtz equation, for example \cite{Mc:00}, \cite{SaSc:11}, \cite{St:08}, \cite{HsWe:08}; we will use \cite{ChGrLaSp:12} as a default reference (since it, like us, is concerned with the high-frequency behaviour of these integral operators).

If $u$ is a solution of the homogeneous Helmholtz equation in $\Oe$
then an application of Green's formula yields
\beq\label{eq:Green}
u(x) = - \int_\Gamma \Phi_k(x,y) \dnpu(y) \, \rd s(y) + \int_\Gamma \pdiff{\Phi_k(x,y)}{n(y)} \gamma_+ u(y) \, \rd s(y), \quad x\in \Oe, 
\eeq
(see, e.g., \cite[Theorem 2.21]{ChGrLaSp:12}), where 
$\Phi_k(\bx,\by)$ is the fundamental solution of the Helmholtz equation given by 
\beq\label{eq:fund}
\Phi_k(\bx,\by)=\displaystyle\frac{\ri}{4}H_0^{(1)}\big(k|\bx-\by|\big), \,\,d=2,\quad\quad \Phi_k(\bx,\by) = \frac{\re^{\ri k |\bx-\by|}}{4\pi |\bx-\by|}, \,\,d=3.
\eeq
Taking the exterior Dirichlet and Neumann traces of \eqref{eq:Green} on $\Gamma$ and using the jump relations for the single- and double-layer potentials (see, e.g., \cite[Equation 2.41]{ChGrLaSp:12} we obtain the following two integral equations
\beq\label{eq:BIE1}
S_k \dnpu = \left( -\half I + D_k\right) \gamma_+ u 
\eeq
and 
\beq\label{eq:BIE2}
\left( \half I + D_k'\right) \dnpu = H_k \gamma_+ u, 
\eeq
where $S_k$, $D_k$ are the single- and double-layer operators, $D_k'$ is the adjoint double-layer operator, and $H_k$ is the hypersingular operator. These four integral operators are defined for $\phi\in\LtG$, $\psi \in \HoG$, and $x\in\Gamma$ by 
\begin{align}\label{eq:SD}
&S_k \psi(\bx) := \int_\Gamma \Phi_k(\bx,\by) \psi(\by)\,\rd s(\by), \qquad
D_k \phi(\bx) := \int_\Gamma \frac{\partial \Phi_k(\bx,\by)}{\partial n(\by)}  \phi(\by)\,\rd s(\by),\\
&D'_k \psi(\bx) := \int_\Gamma \frac{\partial \Phi_k(\bx,\by)}{\partial n(\bx)}  \psi(\by)\,\rd s(\by),\quad
H_k \phi(\bx) := \pdiff{}{n(x)} \int_\Gamma \pdiff{\Phi_k(\bx,\by)}{n(\by)} \phi(\by)\, \rd s(\by).
\end{align}
When $\Gamma$ is Lipschitz, the integrals defining $D_k$ and $D'_k$ 
must be understood as Cauchy principal value integrals and even when $\Gamma$ is smooth there are subtleties in defining $H_k\psi$ for $\psi\in\LtG$ which we ignore here (see, e.g., \cite[\S2.3]{ChGrLaSp:12}). 

\subsection{The Dirichlet problem}

In the case of the Dirichlet problem, the integral equations
\eqref{eq:BIE1} and \eqref{eq:BIE2} are both integral equations for
the unknown Neumann trace $\dnpu$. However \eqref{eq:BIE1} is not
uniquely solvable when $-k^2$ is a Dirichlet eigenvalue of the
Laplacian in $\Oi$, and \eqref{eq:BIE2} is not uniquely solvable
when $-k^2$ is a Neumann eigenvalue of the Laplacian in $\Oi$. (This is
because if $w$ solves the \emph{interior} Helmholtz equation, Green's
formula yields
$$
\left( \frac 12 I +D_k \right) \gamma_- w  = S_k \pa_n^- w;
$$
hence existence of nullspace of these operators is equivalent to
existence of Dirichlet/Neumann eigenvalues.)

The standard way to resolve this difficulty is to take a linear combination of the two equations, which yields the integral equation
\beq\label{eq:CFIE2}
\opA \dnpu = \opBM \gamma_+ u 
\eeq
where 
\beq\label{eq:CFIEdef}
\opA := \half I + D_k' - \ri \eta S_k
\eeq
and 
\beq\label{opBMdefinition}
\opBM := H_k + \ri \eta \left(\half I - D_k\right).
\eeq
If $\eta \in \Rea\setminus\{0\}$ 
then the integral operator $\opA$ is invertible (on appropriate Sobolev spaces) and so \eqref{eq:CFIE} can then be used to solve the exterior Dirichlet problem for all (real) $k$.
Furthermore one can then show that if $\eta\in \Rea\setminus\{0\}$ then $\opA$ is a bounded invertible operator from $H^s(\Gamma)$ to itself for $-1\leq s\leq 0$; \cite[Theorem 2.27]{ChGrLaSp:12}.

For the general exterior Dirichlet problem it is natural to pose Dirichlet data in $\HhG$ (since $\gamma_+ u \in \HhG$). The mapping properties of $H_k$ and $D_k$ (see \cite[Theorems 2.17, 2.18]{ChGrLaSp:12}) imply that $\opBM: H^{s+1}(\Gamma) \rightarrow H^s(\Gamma)$ for $-1\leq s\leq 0$, and thus $\opBM \gamma_+ u \in \HmhG$. This indicates that we should consider \eqref{eq:CFIE2} as an equation in $\HmhG$. 

Unfortunately evaluating the $\HmhG$ inner product numerically is expensive, and thus it is not practical to implement the Galerkin method on \eqref{eq:CFIE} as an equation in $\HmhG$ (for a short overview of proposed solutions to this problem, see \cite[\S2.11]{ChGrLaSp:12})
Fortunately, we can bypass this problem in the case of plane-wave or point-source scattering. Indeed, in this case $\gamma_+ u \in \HoG$ and $\dnpu\in\LtG$ \cite[Theorem 2.12]{ChGrLaSp:12}. Since $\opBM \gamma_+u$ and $\opA \dnpu$ are then in $\LtG$, we can consider \eqref{eq:CFIE2} as an equation in $\LtG$, which is a natural space for implementing the Galerkin method.

\subsection{The Neumann problem}\label{section:Neumann}

In the case of the Neumann problem we can view \eqref{eq:CFIE2} as an equation to be solved for $\gamma_+u$. Indeed, given $\dnpu\in\HmhG$, we have $\opA \dnpu\in \HmhG$ and $\opBM \gamma_+u\in \HmhG$. The equation \eqref{eq:CFIE2} can then be cast as the variational problem on $\HhG$: find $\phi\in\HhG$ such that 
\beqs
\langle \opBM \phi, \psi\rangle_\Gamma = \langle \opA \dnpu, \psi\rangle_\Gamma \quad\tfa \psi \in\HhG,
\eeqs
where recall that $\langle\cdot,\cdot\rangle_\Gamma$ is the duality pairing between $H^{-s}(\Gamma)$ and $H^s(\Gamma)$ for $0\leq s\leq 1$.

Although this gives a practically-realizable Galerkin method, the fact that $\opBM$ is a first-kind operator means that the condition number of the discretized system depends on the discretization and thus it is desirable to precondition the equation with an operator of opposite order before discretizing (see, e.g., \cite[\S13]{St:08} for a discussion of this technique in general).

For $\opBM$ this strategy amounts to multiplying \eqref{eq:BIE2} by an operator $R: H^{-1}(\Gamma) \rightarrow \LtG$ and then adding it to $-\ri \eta$ multiplied by \eqref{eq:BIE1}. This results in the equation

\beq\label{eq:CFIE3}
\opBMt \gamma_+ u = \opAt \dnpu
\eeq
where 
\beqs
\opBMt:= RH_k + \ri \eta \left( \half I - D_k\right)
\eeqs
and
\beqs
\opAt := R\left(\half I + D_k'\right) - \ri \eta S_k. 
\eeqs
The mapping properties of $R$ and the boundary integral operators $S_k, D_k, D'_k, H_k$ imply that both $\opBMt$ and $\opAt$ are bounded operators mapping $\LtG$ to itself, and thus, in the case when $\dnpu\in\LtG$, \eqref{eq:CFIE3} can be considered as an integral equation in $\LtG$. Of course, $R$ must satisfy some additional conditions to ensure that \eqref{eq:CFIE3} has a unique solution for all $k>0$. 

The most common choice is to take $R=S_0$, motivated by the Calderon identity
\beqs
S_0 H_0 = -\half I + D_0^2
\eeqs
(\cite[Equation 2.56]{ChGrLaSp:12}) and the fact that $S_0(H_k-H_0)$ is compact (since $H_k-H_0$ has a weakly singular kernel; see \cite[Equation 2.25]{ChGrLaSp:12}). 

The choice $R= S_{\ri k}$ was proposed in \cite{BrElTu:12}, and
further used and analyzed in, e.g., \cite{BoTu:13},
\cite{ViGrGi:14}. 
Other choices for $R$ include principal symbols of certain pseudodifferential operators \cite{BoTu:13}, and (for the indirect analogue of \eqref{eq:CFIE3}) approximations of the NtD map \cite[\S8]{AnDa:11}.

\subsection{Decompositions of inverses of combined potential operators}\label{sec:6-3}

The decomposition \eqref{eq:key} of $\opAinv$ in terms of $\DtN$ and $\ItD$ is implicit in much of the work on $\opA$, but was first written down explicitly in \cite[Theorem 2.33]{ChGrLaSp:12}, along with the analogous decomposition for $\opBM^{-1}$ (as a special case of the decomposition of the inverse of the integral operator for the exterior impedance problem).

In Lemma \ref{lemma:inverses} below we provide an alternative, more intuitive, proof of these decompositions. We also give the analogous decomposition of the operator $\opBMt^{-1}$ in terms of $\NtD$ and $\ItDR$, where the operator $\ItDR :\LtGt$ maps $g\in \LtG$ to the Dirichlet trace of the solution of the BVP
\beqs
\Delta u +k^2 u =0 \quad \tin \Oi, \qquad R\dnmu - \ri \eta \gamma_-u = g \quad \ton \Gamma
\eeqs
(assuming appropriate conditions on $R$ are imposed so that this BVP has a unique solution for all $k>0$).

\begin{lemma}\label{lemma:inverses}
We have the following expressions for the inverses of
combined-potential operators:
\begin{align}
\opAinv &= I - (\DtN - \ri \eta ) \ItD,\\
(\opBM)^{-1} &= \NtD-(I-\ri\eta \NtD) \ItD,\\
(\opBMt)^{-1} &= \NtD R^{-1} - (I-\ri \eta \NtD R^{-1}) \ItDR.\label{eq:611}
\end{align}
\end{lemma}

\begin{proof}[Proof of Lemma \ref{lemma:inverses}]
We recall (e.g.\ from Section~2.5 of \cite{ChGrLaSp:12}) the formula for
the interior and exterior \emph{Calder\'on projectors}, which project
onto pairs of Dirichlet and Neumann data for solutions to the
Helmholtz equation in $\Omega_-$ and $\Omega_+$ (with radiation
condition) respectively.  In terms of layer potentials, we may write
these operators as
$$
\Pi_\pm=\frac 12 I\pm M_k,\quad M_k \equiv \bpm D_k &
  -S_k\\ H_k & -D_k' \epm
$$
(Here we have departed from the notation of \cite{ChGrLaSp:12} for the
Calder\'on projectors---these authors use $P_\pm$---as the letter $P$
is somewhat overloaded.)

These definitions imply that
$$
\bpm -\ri\eta & 1 \epm \cdot \Pi_-= \bpm
  -\opBM & \opA \epm.
$$
Hence
\beq\label{eq:612}
\bpm -\ri\eta & 1 \epm \cdot \Pi_- \bpm \phi \\ \psi \epm = g
\Longleftrightarrow -\opBM\phi +\opA \psi = g.
\eeq

On the other hand, since $\Pi_-$ projects to Cauchy data for the
interior Helmholtz problem, we assuredly find that 
\begin{equation}\label{combinedequation}
\bpm -\ri\eta & 1 \epm \cdot \Pi_- \bpm \phi \\ \psi \epm= g
\end{equation}
means that 
$$
\Pi_- \bpm \phi \\ \psi \epm
$$
are Cauchy data for the interior impedance problem, hence we may
rewrite
$$
\Pi_- \bpm \phi \\ \psi \epm=\bpm \ItD(g)\\ \ItN(g) \epm.
$$
Since $\Pi_++\Pi_-=I,$ we now find that
$$
\Pi_+ \bpm \phi \\ \psi \epm=\bpm \phi-\ItD(g)\\ \psi-\ItN(g) \epm.
$$
Note that the RHS is now guaranteed to be Cauchy data for a solution
of the exterior Helmholtz equation (with radiation condition) and
hence we may write its two components in terms of one another via the
maps $\DtN$ and $\NtD.$

Now we split into the special cases of $\phi=0$ or $\psi=0.$  In the
former case we have
$$
\Pi_+ \bpm 0 \\ \psi \epm=\bpm -\ItD(g)\\ -\DtN(\ItD(g)) \epm
$$
(where we have written the second component in terms of the first using $\DtN$).
Thus
\begin{align*}
\psi&= \bpm -\ri\eta & 1 \epm \cdot \bpm 0 \\ \psi \epm\\ &= \bpm -\ri\eta
& 1 \epm
\cdot (\Pi_++\Pi_-) \bpm 0 \\ \psi \epm
\\ &=\bpm -\ri\eta & 1 \epm \cdot \bpm -\ItD(g)\\ -\DtN(\ItD(g)) \epm+g
\end{align*}
where we have used \eqref{combinedequation} to evaluate the $\Pi_-$
term.
Likewise, when $\psi=0$ we have
$$
\Pi_+ \bpm \phi \\ 0 \epm=\bpm \phi-\ItD(g)\\ \DtN(\phi -\ItD(g)) \epm.
$$
Thus
\begin{align*}
-\ri\eta \phi&= \bpm -\ri\eta & 1 \epm \cdot \bpm \phi \\ 0 \epm\\ &= \bpm
-\ri\eta & 1 \epm
\cdot (\Pi_++\Pi_-) \bpm \phi \\ 0 \epm
\\ &=\bpm -\ri\eta & 1 \epm \cdot \bpm \phi-\ItD(g)\\ \DtN(\phi-\ItD(g)) \epm+g
\end{align*}
In both cases, solving for $\psi$ (respectively $\phi$) and recalling \eqref{eq:612} gives the desired
expression in terms of $g$ (in the latter case, we use that
$\phi=\NtD \circ \DtN \phi$).

Finally, to obtain the formula for $\opBMtinv,$ we apply the same
argument as for $\opBMinv,$ but where we consider
$$
\bpm -\ri\eta & R \epm \cdot \Pi_-
$$
throughout, rather than $$\bpm -\ri\eta & 1 \epm \cdot \Pi_-.$$
\end{proof}

The estimate $\opBM^{-1}$ analogous to the estimate 
\eqref{eq:Ainv_bound_main} on $\opAinv$ is as follows.
\begin{lemma}\label{thm:CFIE-Neumann}
Let $\Oe\subset \Rea^d$, $d=2,3$, be a smooth, nontrapping domain and suppose that $\eta$ satisfies Assumption \ref{ass:eta}. 
Then, given $k_0>0$,
\beq\label{eq:Binv_bound_main}
\norm{\opBM^{-1}}_{\LtG\rightarrow H^1_k(\Gamma)} \lesssim \abs{k}^{1-\beta}
\eeq
for all $\abs{k}\geq k_0$, where $\beta$ is as in Theorem~\ref{thm:2}.
\end{lemma}
Since this integral operator is
not used in practice, however (as explained in \S\ref{section:Neumann}), we
do not include the proof.
Note that an estimate from $\HmhG$ to $\HhG$ can be obtained from \eqref{eq:Binv_bound_main} by interpolation.

The decomposition of $\opBMt^{-1}$ given by \eqref{eq:611} below and the sharp bounds on $\NtD$ in Theorem \ref{thm:2} reduce the problem of bounding $\|\opBMtinv\|_{\LtGt}$ to that of bounding $\ItDR$ for the different choices of $R$, however we do not pursue this further here.


\section{Concluding remarks: the conditioning of $\opA$}\label{sec:condition}

In \S\ref{sec:1-3} we stated that the present paper combined with the recent work of 
Galkowski--Smith and Galkowski
almost completes the study of the high frequency behaviour of $\normA$ and $\normAinv$, and thus of the condition number
\beq\label{eq:cond_def_again}
\cond(\opA) := \normA_{\LtGt} \normAinv_{\LtGt}.
\eeq
We conclude this paper by justifying this remark in \S\ref{sec:7-1}, but then also questioning in \S\ref{sec:7-2} whether the condition number is an appropriate object to study in relation to $\opA$.

\subsection{Upper bounds on $\cond(\opA)$}\label{sec:7-1}

We begin by recalling the recent sharp bounds on $\|S_k\|_{\LtGt}$ and $\|D_k\|_{\LtGt}$ proved in \cite[Theorem 2]{GaSm:14}, \cite[Theorem A.1]{HaTa:14}. (Note that $\|D_k\|_{\LtGt}= \|D'_k\|_{\LtGt}$, and so these bounds are sufficient to bound $\normA_{\LtGt}$.)

\begin{theorem}\textbf{\emph{(\cite[Theorem 1.2]{GaSm:14}, \cite[Theorem A.1]{HaTa:14}, \cite[Theorem 4.4]{Ga:15a})}}\label{thm:Gal}
With $\Oi$ and $\Gamma$ defined in \S\ref{sec:1-1}, if $\Gamma$ is a finite union of compact embedded $C^\infty$ hypersurfaces then there exists $k_0$ such that, for $k\geq k_0$,
\beqs
\N{S_k}_{\LtGt} \lesssim k^{-1/2}\log k, \quad \N{D_k}_{\LtGt} \lesssim k^{1/4}\log k.
\eeqs
If $\Gamma$ is a finite union of compact subsets of $C^\infty$ hypersurfaces with strictly positive curvature, then 
\beqs
\N{S_k}_{\LtGt} \lesssim k^{-2/3}\log k, \quad \N{D_k}_{\LtGt} \lesssim k^{1/6}\log k.
\eeqs
Moreover, modulo the factor $\log k$, all of the estimates are sharp.
\end{theorem}

Note that in 2-d the sharp bound $\|S_k\|_{\LtGt}\lesssim k^{-1/2}$ was proved in \cite[Theorem 3.3]{ChGrLaLi:09}.

Combining these bounds with the bounds on $\normAinv$ \eqref{eq:Ainv_CWM} and \eqref{eq:Ainv_bound_main}, as well as bounds when $\Gamma$ is the circle or sphere obtained by \cite{Gi:97}, \cite{DoGrSm:07}, \cite{BaSa:07} (see the review in \cite[\S5.4]{ChGrLaSp:12}) we obtain the following theorem.

\begin{theorem}[Upper bounds on the condition number]\label{thm:condition}

\

\noi (a) Let $\Oi$ be star-shaped with respect to a ball, with $\Gamma$ piecewise smooth.
When $d=2$, if 
\beqs
k^{3/4}\log k \lesssim |\eta|\lesssim k
\eeqs
then
\beq\label{eq:71}
\cond (\opA) \lesssim k^{1/2}.
\eeq
When $d=3$, if 
\beqs
k^{3/4} \lesssim |\eta|\lesssim k
\eeqs
then
\beq\label{eq:71a}
\cond (\opA) \lesssim k^{1/2} \log k.
\eeq
(b) If $\Oi$ is nontrapping and $\eta$ satisfies Assumption \ref{ass:eta} (which includes the case $|\eta|\sim k$), then \eqref{eq:71} holds when $d=2$ and \eqref{eq:71a} holds when $d=3$.

\noi (c) If $\Oi$ is star-shaped with respect to a ball, $\Gamma$ is the finite union of smooth surfaces with strictly positive curvature, and 
\beqs
k^{5/6} \lesssim |\eta|\lesssim k
\eeqs
then
\beq\label{eq:72}
\cond (\opA) \lesssim k^{1/3}\log k.
\eeq
In particular, if $\Oi$ is a 2- or 3-d ball (i.e., $\Gamma$ is the circle or sphere) then $\cond (\opA) \lesssim k^{1/3}$ when 
\beqs
k^{2/3}\lesssim |\eta|\lesssim k.
\eeqs 
\end{theorem}

Earlier we stated that this theorem ``almost completes" the study of $\cond (\opA)$. One thing that is missing is a lower bound on $\cond(\opA)$ that shows the choice $|\eta|\sim k$ is optimal.
Indeed, in 2-d, if $\Gamma$ contains a straight line segment, then by \cite[Theorem 4.2]{ChGrLaLi:09}
\beqs
\N{\opA}_{\LtGt} \gtrsim \frac{|\eta|}{k^{1/2}} + \cO \left( \frac{|\eta|}{k}\right) + 1 
\eeqs
as $k\tendi$, uniformly in $|\eta|$. The only existing lower bound on $\normAinv$ is $\normAinv \geq 2$, which holds if a part of $\Gamma$ is $C^1$ \cite[Lemma 4.1]{ChGrLaLi:09}, and with this alone we cannot rule out the possibility that $\cond(\opA) \ll k^{1/2}$ for a choice of $|\eta|\ll k$ but $\gtrsim k^{3/4}\log k$ (although we do not expect this to be the case).


\subsection{Should we really be interested in the condition number?}\label{sec:7-2}

To be concrete, we consider solving numerically the integral equation \eqref{eq:CFIE} (as an equation in $\LtG$) via the Galerkin method, i.e.~given a sequence of finite-dimensional nested subspaces $V_N\subset \LtG$, we seek $v_N\in V_N$ such that
\beq\label{eq:721}
( \opA v_N, w_N)_{\LtG} = (f_{k,\eta}, w_N)_{\LtG} \quad \tfa w_N\in V_N.
\eeq
We restrict attention to the case when $V_N$ consists of piecewise polynomials (and so we do not consider, e.g., subspaces involving oscillatory basis functions; see, e.g., \cite{ChGrLaSp:12} and the references therein), and furthermore we only consider the $h$-boundary element method (BEM) (i.e., the piecewise polynomials have fixed degree but decreasing mesh width $h$). 

Given a basis of $V_N$, equation \eqref{eq:721} becomes a system of linear equations; for simplicity we do not consider preconditioning this system.

For the high-frequency numerical analysis of this situation, there are now, roughly speaking, two tasks:
\ben
\item We expect that the subspace dimension $N$ ($\sim h^{-(d-1)}$) must grow with $k$ in order to maintain accuracy, and we would like $k$- and $\eta$-explicit bounds on the required growth.
\item One usually solves the linear system with an iterative solver such as the generalized minimal residual method (GMRES); we expect the number of iterations required to achieve a prescribed accuracy to increase with $k$, and we would like $k$- and $\eta$-explicit bounds on this growth.
\een

\paragraph{Regarding 1:}
The analysis in \cite{GrLoMeSp:15} shows that there exists a $C>0$ such that if
\beqs
h\left( \N{D_k'}_{\LtG\rightarrow \HoG} + |\eta|\N{S_k}_{\LtG\rightarrow \HoG}\right) \normAinv_{\LtGt} \leq C
\eeqs
then the sequence of Galerkin solutions $v_N$ is quasioptimal (with the constant of quasioptimality independent of $k$), i.e.,
\beqs
\N{\dnpu- v_N}_{\LtG} \lesssim \min_{w_N\in V_N}\N{\dnpu- w_N}_{\LtG};
\eeqs
see \cite[Corollary 4.1]{GrLoMeSp:15}. Therefore, minimizing
\beq\label{eq:722}
\left(\N{D_k'}_{\LtG\rightarrow \HoG} + |\eta|\N{S_k}_{\LtG\rightarrow \HoG}\right)\normAinv_{\LtGt} 
\eeq
gives the least restrictive condition on $h$.

This is not quite the same as minimizing the condition number, but if we believe that the $L^2 \rightarrow H^1$-norms of $D'_k$ and $S_k$ are proportional to the $L^2\rightarrow L^2$-norms (with the same constant of proportionality), as they are in the case of the circle and sphere at least (with ``constant" of proportionality $k$), then minimizing \eqref{eq:722} is equivalent to minimizing the condition number.\footnote{The methods used to prove the bounds in Theorem \ref{thm:Gal} also appear to be able to prove the corresponding $L^2\rightarrow H^1$ bounds with an extra factor of $k$ on the right-hand sides \cite{Ga:15}; thus the proportionality discussed above would hold.}

Two remarks: 
\bit
\item In \cite{GrLoMeSp:15} bounds on the $L^2 \rightarrow H^1$-norms are obtained and it is shown that, if $|\eta|\sim k$ and $\Oi$ is both $C^2$ and star-shaped with respect to a ball, then the quantity in \eqref{eq:722} is bounded by $k^{3/2}$ in 2-d, yielding the condition for quasioptimality $hk^{3/2}\lesssim 1$. In the case of the circle/sphere, better bounds on the norms can be used to obtain the condition for quasioptimality $hk^{4/3}\lesssim 1$. In practice, one sees that the $h$-BEM is quasi-optimal when $hk\lesssim 1$ (i.e., it does not suffer from the pollution effect), see, e.g., \cite[\S5]{GrLoMeSp:15}, but this observation has yet to be established rigorously.
\item Here we have only talked about the $h$-BEM; the $hp$-BEM (where the polynomial degree, $p$, is variable) is less sensitive to the value of $\eta$ and the norms of $\opA$ and $\opAinv$; see \cite{LoMe:11}, \cite{Me:12} for more details.
\eit

\paragraph{Regarding 2:}

In the discussion above we noted that, in practice, $hk\lesssim 1$ is sufficient to ensure $k$-independent quasioptimality of the Galerkin method. Since $N\sim h^{-(d-1)}$, this condition implies that, as $k$ increases, the size of the linear system must grow like $k^{(d-1)}$ to maintain accuracy.
Iterative methods, such as GMRES, are then the methods of choice for solving such large linear systems.

For Hermitian matrices there are well-known bounds on the number of iterations of the conjugate gradient method in terms of the condition number of the matrix \cite[Chapter 3]{Gr:97}, and for normal matrices there are well-known bounds on the number of GMRES iterations in terms of the location of the eigenvalues (which can be rewritten in terms of the condition number) \cite[Theorem 5]{SaSc:86}, \cite[Corollary 6.33]{Sa:03} (how satisfactory these bounds are is another question, but they exist). In contrast, for non-normal matrices it is not at all clear that the condition number tells us anything about the behaviour of GMRES (at least, there do not exist any bounds on the number of iterations in terms of the condition number of non-normal matrices).

As a partial illustration of this in the context of Helmholtz integral
equations, the recent work of Marburg \cite{Ma:14}, \cite{Ma:15a}
has emphasized
that, at least for certain collocation discretizations of the integral
equation \eqref{eq:CFIE2}, used as an integral equation for the Neumann
problem, the sign of $\eta$ affects the number of GMRES iterations (with
$\eta =k$ leading to much smaller iteration counts that $\eta = -k$). An
analogous effect occurs for similar collocation discretizations of the
integral equation \eqref{eq:CFIE2} used as an equation to solve the
Dirichlet problem, with the choice of $\eta=k$ much better than $\eta=-k$
\cite{Ma:15}. In contrast, the condition number estimates in Theorem
\ref{thm:condition} are independent of the sign of $\eta$, suggesting that
the condition number is not the right tool to investigate the behaviour of
GMRES.

A concept that does give bounds on the number of GMRES iterations for non-normal matrices is \emph{coercivity}.
On the operator level (for $\opA$ on $\LtG$), coercivity is the statement that there exists an $\alpha_{k,\eta}>0$ such that
\beqs
|\ang{\opA \phi,\phi}_{\LtG}|\geq \alpha_{k,\eta} \N{\phi}^2_{\LtG} \quad\tfa \phi \in\LtG,
\eeqs
and the matrix of the Galerkin method \eqref{eq:721} then inherits an analogous property (see, e.g., \cite[Equation 1.20]{SpKaSm:15}). If $\opA$ is coercive, then the so-called Elman estimate for GMRES \cite{El:82}, \cite[Theorem 3.3]{EiElSc:83}, \cite[\S1.3.2]{OlTy:14}
can be used to prove a bound on the number of GMRES iterations required to achieve a prescribed accuracy, with the bound given in terms of $\alpha_{k,\eta}$ and $\N{\opA}_{\LtGt}$; see \cite[Equation 1.21]{SpKaSm:15}.

It is not clear whether bounds on the number of GMRES iterations obtained via this method are sharp, and so far $\opA$ has only been proved to be coercive when $\eta \gtrsim k$ and $\Oi$ is strictly convex (and under additional smoothness assumptions on $\Gamma$), so we do not yet know enough to make a provably-optimal choice of $\eta$ via this approach. However, we do know that the sign of $\eta$ \emph{does} matter for coercivity. Indeed, when $\Oi$ is a ball, $\opA$ is coercive when $\eta=k$ \cite{DoGrSm:07}, but not when $\eta=-k$ \cite[\S1.2]{SpKaSm:15}. The dependence of coercivity on the sign of $\eta$ is consistent, therefore, with the results of Marburg that indicate that the number of GMRES iterations for $\opA$ depends on the sign of $\eta$.


\section{Acknowledgements} 

The authors thank Alex Barnett (Dartmouth and Simons Foundation), Charles Epstein (Pennsylvania), Jeffrey Galkowski (Stanford), David Hewett
(Oxford), Steffen Marburg (Universit\"at der Bundeswehr M\"unchen)
Andrea Moiola (Reading), Andr\'as Vasy (Stanford), and Leonardo Zepeda--N\'u\~{n}ez (University of California at Irvine) for helpful conversations.  
The authors also thank the referees for their constructive comments.

The first author gratefully acknowledges the support
of NSF postdoctoral fellowship DMS-1103436.  The second author
gratefully acknowledges the support of EPSRC Grant EP/1025995/1. The
third author gratefully acknowledges the support of NSF Grant
DMS-1265568.


\end{document}